\documentclass{article}%
\usepackage{amsmath}
\usepackage{amsfonts}
\usepackage{amssymb}
\usepackage{graphicx}
\usepackage{epsfig}
\usepackage{color}
\usepackage{tabularx,ragged2e,booktabs,caption}%
\providecommand{\U}[1]{\protect\rule{.1in}{.1in}}

\newcommand{\e}{\ensuremath{\varepsilon}}

\newcommand{\TT}{\ensuremath{\mathbb{T}^2}}
\newtheorem{theorem}{Theorem}

\newtheorem{definition}[theorem]{Definition}

\newtheorem{lemma}[theorem]{Lemma}

\newtheorem{proposition}[theorem]{Proposition}
\newtheorem{remark}[theorem]{Remark}

\newenvironment{proof}[1][Proof]{\noindent\textbf{#1.} }{\ \rule{0.5em}{0.5em}}

\begin{document}

\title{The onset of layer undulations in smectic A liquid crystals due to a strong
magnetic field}
\author{A. Contreras\thanks{Department of Mathematical Sciences,
New Mexico State University,
1290 Frenger Mall,
MSC 3MB. Science Hall 225,
Las Cruces, New Mexico 88003-8001. E-mail: acontre@nmsu.edu}, 
C. Garcia-Azpeitia\thanks{Departamento de Matem\'{a}ticas, Facultad de Ciencias, Universidad Nacional Aut\'{o}noma de M\'{e}xico.
Circuito Exterior S/N, Cd. Universitaria, Colonia Copilco el Bajo,
Delegación Coyoac\'{a}n, M\'{e}xico D.F.,\; M\'{e}xico. \;C.P. 04510. E-mail: cgazpe@ciencias.unam.mx}, 
C.J. Garc{\'\i}a-Cervera\thanks{Department of Mathematics,
South Hall, Room 6707,
University of California,
Santa Barbara, CA 93106-3080.
 USA. E-mail: cgarcia@math.ucsb.edu}, 
S. Joo.\thanks{Department of Mathematics \& Statistics,
Old Dominion University,
2313 Engr. and Comp. Sci. Bldg.,
Norfolk, VA 23529 . Norfolk, VA 23529. USA. E-mail: sjoo@odu.edu}}
\maketitle

\begin{abstract}
We investigate the effect of a strong magnetic field on a three dimensional
smectic A liquid crystal. We identify a critical field above which the uniform
layered state loses stability; this is associated to the onset of layer
undulations. In a previous work \cite{G-J2}, Garc{\'{\i}}a-Cervera and Joo
considered the two dimensional case and analyzed the transition to the
undulated state via a simple bifurcation. In dimension $n=3$ the situation is
more delicate because the first eigenvalue of the corresponding linearized
problem is not simple. We overcome the difficulties inherent to this higher
dimensional setting by identifying the irreducible representations for natural
actions on the functional that take into account the invariances of the
problem thus allowing for reducing the bifurcation analysis to a subspace with
symmetries. We are able to describe at least two bifurcation branches, one of
which is stable, highlighting the richer landscape of energy critical states
in the three dimensional setting. Finally, we analyze a reduced two
dimensional problem, assuming the magnetic field is very strong, and are able
to relate this to a model in micromagnetics studied in \cite{Alouges-Serfaty},
from where we deduce the periodicity property of minimizers.

\end{abstract}

\section{Introduction}

The main thrust of this work is to understand the undulation phenomena in
smectic A liquid crystals caused by a strong magnetic field. The mathematical
framework for our study is the model introduced by de Gennes \cite{dG}. The
Landau de Gennes energy describes the state of the liquid crystal in terms of
a director $n$ and a complex order parameter
\[
\Psi=\rho(x)e^{iq\omega(\mathbf{x})},
\]
for the layered structure. The molecular mass density is defined by
\[
\delta(x)=\rho_{0}(x)+\frac{1}{2}(\Psi(x)+\Psi^{\ast}(x))=\rho_{0}%
(x)+\rho(x)\cos q\omega(x),
\]
where $\rho_{0}$ is a locally uniform mass density, $\rho(x)$ is the mass
density of the smectic layers, and $\omega$ parametrizes the layers so that
$\nabla\omega$ is the direction of the layer normal. Also, $q$ is the wave
number and $2\pi/q$ is the layer thickness. The state $\Psi\equiv0$ corresponds to no layered structure (the sample is in a nematic phase).
If $|\Psi|=\rho$ is a nonzero constant, then the sample is in a smectic state
throughout. The contribution to the energy due to the external
magnetic field is accounted for by the magnetic free energy density%
\[
-\chi_{a}H^{2}(n\cdot h)^{2},
\]
where $\chi_{a}$ is the magnetic anisotropy, $H$ is the magnitude of the
magnetic field and $h$ is a unit vector representing the direction of the
field. We assume that $\chi_{a}<0$. As a consequence, the director has a
preferred orientation perpendicular to the direction of the applied magnetic
field. \footnote{In \cite{G-J2}, $\chi_{a}$ is assumed to be positive; in that
case the preferred orientation for the director is parallel to the field which
fully specifies the orientation in that $2d$ setting.}

In the one constant approximation case for the Oseen-Frank nematic energy, the
total free energy for smectic A is given by,
\begin{equation}
F_{0}(\Psi,n)=\int_{\Omega}C\left\vert \nabla\Psi-iqn\Psi\right\vert
^{2}+K\left\vert \nabla n\right\vert ^{2}+\frac{g_{0}}{2}\left(  \left\vert
\Psi\right\vert ^{2}-\frac{r}{g_{0}}\right)  ^{2}-\chi_{a}H^{2}(n\cdot
h)^{2}\text{,} \label{mcG}%
\end{equation}
where the material parameters $C,K,g_{0}$ and temperature dependent parameter
$r=T_{NA}-T$ are fixed positive constants and
\[
\Omega=(-L_{1},L_{1})\times(-L_{2},L_{2})\times(-d_{0},d_{0})
\]
is a rectangular box. This models a smectic A liquid crystal with smectic
layers parallel to the {top and bottom }boundaries.

In our study we assume that the magnetic field is applied in the
$z$-direction, perpendicular to the layers, that is $h=e_{3}.$With the
director constraint $|n|=1$, we rewrite the magnetic energy density
\[
f_{m}=-|\chi_{a}|H^{2}(n_{1}^{2}+n_{2}^{2}),
\]
by omitting the constant term.

\subsection{Nondimensionalization}

We define the dimensionless order parameter, $\psi= \sqrt{\frac{g_{0}}{r}}
\Psi$, and do the change of variables $\bar{\mathbf{x}}=\mathbf{x}/d$ with
$d=2d_{0}/\pi$, and obtain the following nondimensionalized energy
\begin{equation}
F(\psi,n):=\frac{\varepsilon}{dK}F_{0}(\psi,n)=\int_{\Omega}\frac
{1}{\varepsilon}\left\vert c\,\varepsilon\nabla\psi-in\psi\right\vert
^{2}+\varepsilon\left\vert \nabla n\right\vert ^{2}+\frac{g}{2\varepsilon
}(1-\left\vert \psi\right\vert ^{2})^{2}-\tau(n_{1}^{2}+n_{2}^{2})\text{,}
\label{energy}%
\end{equation}
where
\begin{equation}
\label{parameters}\varepsilon=\frac{\lambda}{d}\sqrt{\frac{g_{0}}{r}}%
,\quad\lambda=\sqrt{\frac{K}{Cq^{2}}},\quad\tau=\frac{|\chi_{a}|H^{2}%
d^{2}\varepsilon}{K},\quad c=\sqrt{\frac{Cr}{Kg_{0}}},\quad g=\frac{r}{Cq^{2}%
},
\end{equation}
and $\Omega=(-\pi/a,\pi/a)\times(-\pi/b,\pi/b)\times(-\pi/2,\pi/2)$ with
$a=2d_{0}/L_{1}$ and $b=2d_{0}/L_{2}$. The values $d=1mm$ and $\lambda
=20\mathring{A}$ are employed in \cite{dG-P}. The dimensionless parameter
$\varepsilon$ is in fact the ratio of the layer thickness to the sample
thickness. The case $\varepsilon\ll1$ corresponds to a lower dimensional approximation and this is the point of view in the works \cite{G-J2, Tiz, ChengPhi, Gol, RenWei} where the emergence of interesting structures is captured in minimizers via a $\Gamma$-convergence approach. We point out that his kind of variational dimension reduction has already been considered in the analysis of the related Ginzburg-Landau model \cite{CS,C,CL}. One strength of our approach is that we do not require $\varepsilon$ to be small but just that it avoids values at which resonances occur. The other advantage is that the branches we find do not only start near the lowest critical field but at higher bifurcation values as well, so that metastable bifurcations are also obtained.

\subsection{The onset of layer undulations}

We are interested in studying the stability of the uniformly layered state:
\[
(\psi_{0},\mathbf{n}_{0})\equiv(e^{iz/c\varepsilon},\mathbf{e}_{3}).
\]
To that end, we impose periodic boundary conditions for $\psi$ and
$\mathbf{n}$ in the $x$ and $y$ directions, while we assume Dirichlet boundary
conditions for $\psi$ and $\mathbf{n}$ on the boundary plates:
\begin{equation}
\psi(x,y,\pm\pi/2)=e^{i(\pm\pi/2)c\varepsilon}\qquad\mbox{and}\qquad
n(x,y,\pm\pi/2)=\mathbf{e}_{3}\text{,} \label{boundary-n}%
\end{equation}
for all $(x,y)\in\lbrack-\pi/a,\pi/a]\times\lbrack-\pi/b,\pi/b]$. With this
boundary condition, we assume that $\Omega=T^{2}\times I$, where $T^{2}$ is
the two torus identified with $[-\pi/a,\pi/a]\times\lbrack-\pi/b,\pi/b]$ and
$I=[-\pi/2,\pi/2]$.

It is known that the application of an electric or magnetic field of large
enough strength, or bare mechanical tension \cite{GePr, GaCJo, IsLa, SeSmLa,
Si, St98, StSt}, can trigger the emergence of undulations in the smectic
layers: this is called the Helfrich-Hurault effect \cite{He, Hu}. In
\cite{G-J2}, the effect of a magnetic field on a $2d$ smectic A liquid crystal
was considered; a threshold value of the intensity, above which periodic layer
undulations are observed, was estimated and a nontrivial solution curve was
found. The proof relied on Crandall-Rabinowitz bifurcation theory for simple
eigenvalues and thus the solution curve branches off of the lowest eigenvalue. In our situation, the
spectral analysis of the linearized operator about the uniformly layered state
gives a lowest eigenvalue which is not simple and so we cannot follow the same
route to study the instability of the undeformed state at the critical
magnetic field. Instead, we work in a space of maximal symmetries where the
resonances are avoided; This is akin to the Palais principle of symmetric
criticality \cite{Palais}. The previous analysis yields the existence of
branches stemming from the critical magnetic field with the symmetries of
these one dimensional fixed point subspaces of irreducible representations(see
Section \ref{bifurcation} for more details). A similar approach has been used to obtain multi-pole periodic solutions to the Gross-Pitaevskii equation in the plane \cite{CG}; these solutions arise as global bifurcations of a nonlinear Schr\"{o}dinger equation.  

Our main result is the following.

\begin{theorem}
\label{Tmain}There exists a critical magnetic field $\tau_{c}$ such that the
uniformly layered state has two bifurcating solutions starting from $\tau
_{c},$ except for a finite union of hypersurfaces in the set of parameters
$(a,b,\varepsilon).$ One of these branches is stable.
\end{theorem}

In Section \ref{bifurcation} we give more precise information about these
branches, in particular we find their symmetries and in \ref{localestimates}
we provide their local asymptotic profile. We are also able to retrieve other
bifurcation branches not covered by Theorem \ref{Tmain}, namely for the
resonant case $a=b$ (see Proposition \ref{propab} for more details).

Next, we investigate what happens at each cross section of the domain under
the assumption that the magnetic field is large enough to render a planar
configuration. We find this reduced model bears strong resemblance with the
one for micromagnetics considered in \cite{Alouges-Serfaty} and from this we
derive important qualitative information about minimizers in the limit
$\varepsilon\to0,$ including a general compactness statement and a sharp lower
bound that allows us to compare between natural candidates for the optimal
pattern. Finally, we provide some numerical evidence in support of our findings.

We now give an outline of the paper. In the next section we introduce notation
and the functional setting where our bifurcation problem is posed. We then
proceed to study the spectral problem of the linearized operator at the
undeformed state and derive the critical field in Section \ref{spectrum}. In
Section \ref{bifurcation} we study the irreducible representations associated
to $O(2)\times O(2)\times\mathbb{Z}_{2}$ and the corresponding isotropy
groups; from these we obtain two fixed one dimensional subspaces were we find
the aforementioned bifurcation branches. We move to the $2d$ dimensional
problem in Section \ref{planar} and deduce the periodicity of minimizers. In the final section we present the numerical
simulations depicting the $2d$ planar configurations and layers at the onset
of undulations, confirming that the estimate of the critical field from our
theory is consistent with the numerical results.

\section{Setting up the problem}

\label{Setting}We start by introducing some notation which is going to be
essential for the bifurcation analysis.

\subsection{Some notions of actions on functionals}

Let $G$ be a group . We recall that a Hilbert space $\mathcal{H}$ is a
$G$-\textbf{representation}, if there is a morphism of groups given by {the
map $\rho:G\rightarrow GL(\mathcal{H})$}.{ }Also, a set $\Omega$ is said to be
\textbf{invariant} if $\rho(\gamma)x\in\Omega$ for all $\gamma\in G$, $x
\in\Omega$. {An \textbf{irreducible representation} is an invariant subspace
of }$\mathcal{H}${ without a proper invariant subspace.}

For a subgroup $H<G$, we consider the \textbf{fixed point space}%
\[
Fix(H)=\{x\in\mathcal{H}:\rho(\gamma)x=x\text{ for all }\gamma\in H\}\text{.}%
\]
The \textbf{orbit} and the \textbf{isotropy group} of a point{ $x\in
\mathcal{H}$ are defined as}%
\[
Gx=\{\rho(\gamma)x:\gamma\in G\}\text{ and }G_{x}=\{\gamma:\rho(\gamma
)x=x\}\text{ respectively.}%
\]
It is said that a map $f:\mathcal{H}\rightarrow\mathcal{H}$ is $G$%
-\textbf{equivariant} if
\[
f(\rho(\gamma)x)=\rho(\gamma)f(x)\text{.}%
\]
Finally, a map $F:\mathcal{H}\rightarrow\mathbb{R}$ is \textbf{invariant} if
$F(x)=F(\rho(\gamma)x)$.

The following properties are used in this article. We refer the reader to
\cite{IzVi03} for a proof.

\begin{lemma}
{ (\cite{IzVi03})}\label{ActionProp}

\begin{enumerate}
\item If $f:\mathcal{H}\rightarrow\mathcal{H}$ is $G$-equivariant and $x\in
Fix(H)$, then the linear map $f^{\prime}(x)$ is $H$-equivariant.

\item If $\rho(\gamma)$ is an isometry in $\mathcal{H}$ for all $\gamma\in G$,
i.e. $\rho:G\rightarrow O(\mathcal{H)}$, then the gradient of an invariant map
$F$ is equivariant, i.e.
\[
f(x)=\nabla F:\mathcal{H\rightarrow H}%
\]
is $G$-equivariant.

\item The restriction of an $G$-equivariant map $f:\mathcal{H}\rightarrow
\mathcal{H}$ to $Fix(H)$ is well defined, i.e. $f(x)\in Fix(H)$ for all $x\in
Fix(H)$.
\end{enumerate}
\end{lemma}

The previous definitions and statement are adapted {for a }manifold
$\mathcal{M}\subset\mathcal{H}$, if $\mathcal{M}$ is invariant by the linear
transformation $\rho(\gamma)$. In this case {the action is }%
\[
\rho(\gamma)u={\rho(\gamma,u)}:G\times\mathcal{M}\rightarrow\mathcal{M}{.}%
\]

\subsection{Natural actions on the Landau-de Gennes energy}

The Landau-de Gennes energy $F$ introduced in \eqref{energy} is well defined
in the manifold
\[
\mathcal{M}=\left\{  (\psi,n)\in H^{2}(T^{2}\times I;\mathbb{C}\times
S^{2}):(\psi,n)(x,y,\pm\pi/2)=(e^{\pm i\pi/2c\varepsilon},\mathbf{e}%
_{3})\right\}  \text{.}%
\]
Since $\mathcal{M}$ is contained in the Hilbert space $L^{2}(T^{2}\times
I;\mathbb{C}\times\mathbb{R}^{2})$, we define the action of the group%
\[
G:=O(2)\times O(2)\times\mathbb{Z}_{2}%
\]
in this Hilbert space as follows. \vskip.1in \textbf{1.} The rotations.
\vskip.1in For $(\varphi,\theta)\in\lbrack0,2\pi)^{2}=O(2)\times O(2),$%

\[
\rho(\varphi,\theta)(\psi,n)=(\psi,n)(x+\varphi/a,y+\theta/b,z)\text{,}%
\]
\vskip.1in \textbf{2.} The reflections. \vskip.1in The action of the planar
reflections $\kappa_{x},\kappa_{y}\in O(2)\times O(2)$ is given by%
\begin{align*}
\rho(\kappa_{x})(\psi,n)  &  =(\psi,R_{x}n)(-x,y,z)\text{,}\\
\rho(\kappa_{y})(\psi,n)  &  =(\psi,R_{y}n)(x,-y,z)\text{,}%
\end{align*}
where $R_{x}=diag(-1,1,1)$ and $R_{y}=diag(1,-1,1) .$

The above defines the action of the direct product of the two copies of the
orthogonal group. \vskip.1in \textbf{3.} The $\mathbb{Z}_{2}$ component.
\vskip.1in We define the action of the reflection $\kappa_{z}\in\mathbb{Z}%
_{2}$ by
\[
\rho(\kappa_{z})(\psi,n)=(\bar{\psi},R_{x}R_{y}n)(x,y,-z)\text{.}%
\]

We have:

\begin{lemma}
\label{invoff} The map $F:\mathcal{M}\rightarrow\mathbb{R}$ is $G$-invariant.
\end{lemma}

\begin{proof}
Since $F$ is invariant by translations and reflections in $x$ and $y$, due to
the boundary conditions, then it is invariant by the action of $O(2)\times
O(2)$. Moreover, it is clear that all the terms in $F$ are invariant by the
action of $\kappa_{z}$ except maybe for $\left\vert c\,\varepsilon\nabla
\psi-in\psi\right\vert ^{2}$. To see that this term is invariant too, let us
denote $(\tilde{\psi},\tilde{n})=\rho(\kappa_{z})(\psi,n)$, then $\nabla
\tilde{\psi}=R_{z}\nabla\bar{\psi}(x,y,-z)$ and using that $R_{z}=-R_{x}R_{y}$
we have%
\begin{align*}
\left\vert c\,\varepsilon\nabla\tilde{\psi}-i\tilde{n}\tilde{\psi}\right\vert
^{2}  &  =\left\vert c\,\varepsilon R_{z}\nabla\bar{\psi}(x,y,-z)-iR_{x}%
R_{y}n\bar{\psi}\right\vert ^{2}\\
&  =\left\vert c\,\varepsilon R_{z}\nabla\psi(x,y,-z)-iR_{z}n\psi\right\vert
^{2}=\left\vert c\,\varepsilon\nabla\psi-in\psi\right\vert ^{2}\text{.}%
\end{align*}
Therefore, we conclude $F$ is invariant under the action of the group $G$.
\end{proof}

\subsection{The linearization at the uniformly layered state}

The uniformly layered state
\[
x_{0}=(\psi_{0},\mathbf{n}_{0})\equiv(e^{iz/c\varepsilon},\mathbf{e}_{3}).
\]
is a trivial critical point of $F$ in $\mathcal{M}$. Since all the elements of
$G$ leave this critical point invariant, their orbits are trivial and their
isotropy groups are all of $G$, i.e. $Gx_{0}=\{x_{0}\}$ and $G_{x_{0}}=G$ for
all $x_{o}$ in $\mathcal{M}$.{ }

We may {parameterize a submanifold of $\mathcal{M}$ by writing%
\begin{equation}
\psi=\psi_{0}(1+iv/c\varepsilon)\text{ and }n(x,y,z)=(w,1)/\left\vert
(w,1)\right\vert , \label{cif}%
\end{equation}
where $v\in H_{0}^{2}(T^{2}\times I;\mathbb{R})$ and$\ w\in H_{0}^{2}%
(T^{2}\times I;\mathbb{R}^{2})$}, where hereafter the subindex $0$ means a
{zero boundary conditions on $v(x,y,\pm\pi/2)=\mathbf{0}$ and $w(x,y,\pm
\pi/2)=0$.}

The map $(w,v)\rightarrow(\psi,n)$ defines { a diffeomorphism} from the space
\[
\mathcal{H}=H_{0}^{2}(T^{2}\times I;\mathbb{R}^{3})
\]
into the submanifold
\[
\mathcal{N}=\{(\psi,n)\in\mathcal{M}:\psi-\psi_{0}\in i\mathbb{R}\text{,
}n(x,y,z)\in S_{+}^{2}\}\text{.}%
\]
This is the set of functions in $\mathcal{M}$ such that {$n$ takes values} in
the upper half {of the unit} sphere and $\psi-\psi_{0}$ is purely imaginary.

Let us define $u=(w,v)\in\mathcal{H}$, writing the energy functional $F$ using
this parametrization of $\mathcal{N},$ we are led to the map $\tilde
{F}:\mathcal{H}\rightarrow\mathbb{R}$ defined by%
\[
\tilde{F}(u):=F(\psi(u),n(u))\text{,}%
\]
thus {$u=0\ $corresponds to the} uniformly layered state $(\psi_{0},n_{0})$,
see (\ref{cif}).

The action of $G$ in these coordinates $u=(w,v)$ is given by%
\[
\rho(\varphi,\theta)u=u(x+\varphi/a,y+\theta/b,z)
\]
and
\[
\rho(\kappa_{x})u=R_{x}u(-x,y,z)\text{, }\rho(\kappa_{y})u=R_{y}%
u(x,-y,z)\text{ and }\rho(\kappa_{z})u=-u(x,y,-z).
\]

The first variation $\delta\tilde{F} (u)$ is characterized by the relation%

\[
\delta\tilde{F}(u)w=\left\langle \nabla\tilde{F}(u),w\right\rangle _{L^{2}}%
\]
{ We define the map} $f:\mathcal{H}\rightarrow L^{2}$ by
\[
f(u)=\nabla\tilde{F}(u)\text{,}%
\]
which is $G$-equivariant, because the action of $G$ is isometric, by property
$2$ of Lemma \ref{ActionProp}.

In order to obtain the second variation of the energy (\ref{energy}) around
the flat layer state $(\psi_{0},n_{0})$ we follow the definition and
derivation of the critical magnetic field in \cite{L-P}. In these coordinates,
consider a perturbation $(\psi_{t},n_{t})=(\psi(tu),n(tu))$, i.e. we have%
\[
\psi_{t}=\psi_{0}+i\frac{\psi_{0}}{c\varepsilon}tv\text{, }\qquad n_{t}%
=\frac{e_{3}+(tw,0)}{|e_{3}+(tw,0)|}=n_{0}+tn_{1}+t^{2}n_{2}+\mathcal{O}%
(t^{3}),
\]
where $n_{0}=e_{3}$, $n_{1}=(w,0)$ and $n_{2}=-\frac{1}{2}\left\vert
w\right\vert ^{2}e_{3}$, since $e_{3}$ and $w$ are orthogonal. Then, the
second variation around the flat layer state $u=0$ is%
\begin{equation}
\delta^{2}\tilde{F}(0)u=\frac{1}{2}\int_{\Omega}\frac{1}{\varepsilon
}\left\vert \nabla v-(w,0)\right\vert ^{2}+\varepsilon\left\vert \nabla
w\right\vert ^{2}-\tau\left\vert w\right\vert ^{2}\text{.}
\label{LinearFunctional}%
\end{equation}
Now, since we have that $\delta\tilde{F}(u)w=\left\langle \nabla\tilde
{F}(u),w\right\rangle $, then
\[
\delta^{2}\tilde{F}(0)u=\left\langle f^{\prime}(0)u,u\right\rangle \text{.}%
\]

From the associated Euler-Lagrange equations, we have that the second
variation at the flat layered state applied to $u=(w,v)$ is%
\[
\delta^{2}\tilde{F}(0)u=\frac{1}{2}\int_{\Omega}\frac{1}{\varepsilon
}\left\vert \nabla v-(w,0)\right\vert ^{2}+\varepsilon\left\vert \nabla
w\right\vert ^{2}-\tau\left\vert w\right\vert ^{2}=\left\langle
Lu,u\right\rangle _{L^{2}}\text{,}%
\]
then the operator $L=f^{\prime}(0):\mathcal{H}\rightarrow L^{2}.$ is given by
\begin{equation}
L(w,v)=(-\varepsilon\Delta w-\frac{1}{\varepsilon}\nabla_{\parallel}v+\frac
{1}{\varepsilon}w-\tau w,-\frac{1}{\varepsilon}\Delta v+\frac{1}{\varepsilon
}\nabla_{\parallel}\cdot w)\text{,}%
\end{equation}
where $\nabla_{\parallel}=(\partial_{x},\partial_{y})$.

Of course, the linearized operator $L$ depends on the parameter $\tau$ and in the next subsection we make this dependence explicit and write $L=L(\tau)$ when we present the bifurcation theorem needed for our main result.  

\subsection{An abstract bifurcation theorem}

 The solutions described in Theorem \ref{Tmain} correspond to critical points of $\tilde{F}:\mathcal{H}\rightarrow
\mathbb{R}.$ We obtain these solutions as zeros of the map%
\[
f(u)=f(u;\tau)=L(\tau)u+O(\left\vert u\right\vert_{\mathcal{H}} ^{2}):\mathcal{H}\rightarrow
L^{2}\text{.}%
\]
Moreover, since $f(u)$ is $G$-equivariant, then the map $f^{H}:Fix(H)\cap
\mathcal{H\times}\mathbb{R}\rightarrow Fix(H)$ is well defined by property 3
of Lemma \ref{ActionProp}, where%
\[
Fix(H)=\{u\in L^{2}:\rho(\gamma)u=u\text{ for }\gamma\in H\}\text{.}%
\]

Therefore, we may consider 
\[
f^{H}(u;\tau)=L^{H}(\tau)u+O(\left\vert u\right\vert_H^{2}),
\]
where $L^{H}$ is the restriction of $L$ to $Fix(H)$ and $\vert\cdot\vert_H$ is the restriction of $\vert \cdot\vert_{\mathcal{H}}$ to $Fix(H).$ To obtain our bifurcation branches we apply Theorem \ref{T5} below to $f^H$ for suitable choices of $H.$

\begin{definition}
Assume that for some linear operator $\mathcal{L}$ and some $H<G,$  $\mathcal{L}^{H}(\tau)$ is a self-adjoint Fredholm operator, $\ker
\mathcal{L}^{H}(\tau_{0})$ is non trivial and $\mathcal{L}^{H}(\tau)$ is invertible for $\tau$
close but different from $\tau_{0}.$ Then we can define $n^{H}(\tau)$ as the
Morse index of $\mathcal{L}(\tau)$ restricted to $\ker \mathcal{L}^{H}(\tau_{0})$ for $\tau$ close
to $\tau_{0}$ and $\eta^{H}(\tau_{0})$ as the net crossing number of
eigenvalues  of $\mathcal{L}^{H}$ at $\tau_{0}$, that is
\[
\eta^{H}(\tau_{0}):=\lim_{\varepsilon\rightarrow0}\left(  n^{H}(\tau
_{0}-\varepsilon)-n^{H}(\tau_{0}+\varepsilon)\right)  \text{.}%
\]

\end{definition}

\begin{theorem}\label{T5}Let $g:\mathcal{H}\to L^2$ be a family of  $G$-equivariant maps (indexed by $\tau$) satisfying
\[g(u;\tau)=\mathcal{L}(\tau)u+\mathcal{O}(\vert u\vert^2_{\mathcal{H}})\] 
\label{BifTheo}If $\mathcal{L}^{H}(\tau)$ is a self-adjoint Fredholm operator and
$\eta^{H}(\tau_{0})$ is odd, then $g^{H}$ has a local bifurcation of zeros
from $(0,\tau_{0})$ in $Fix(H)\cap\mathcal{H\times}\mathbb{R}$.
\end{theorem}

\begin{proof}
We may adapt the Theorem 3.5.1 in \cite{Ni2001} that uses degree theory to
prove that the bifurcation branch form a local continuum. In the case
$\dim\ker L^{H}(\tau_{0})=1$ and $\eta^{H}(\tau_{0})=\pm1$, the theorem can be
obtained with an aplication of the implicit function theorem. This is the
condition of having a simple bifurcation which allows us to give estimates in the
local branches parametrized by the amplitude, using the bifurcation equation as in \cite{CrandallRab}.
\end{proof}
\begin{remark}
Actually, applying a result \cite{HKiel} based on Conley index, we can
prove that if $\eta^{H}(\tau_{0})$ is different from zero, there is local
bifurcation from $(0,\tau_{0})$. However, this local branch does not need to be
a continuum.
\end{remark}

\section{The spectrum: identifying the critical field}

\label{spectrum}

For the bifurcation analysis we need to characterize the spectrum of the
linear map
\[
Lu=f^{\prime}(0)u.
\]

As a starting point we already know that the space $u\in\mathcal{H}$ has a
spectral representation in Fourier series given by
\[
u=\sum_{m,n\in\mathbb{Z}}u_{m,n}(z)e^{i(amx+bny)}\text{,}%
\]
where $u_{m,n}(z)\in H_{0}^{2}(I;\mathbb{C}^{3})$ with $u_{-m,-n}=\bar
{u}_{m,n}$. Next we proceed to further decompose the $z$-dependent Fourier
coefficients $u_{m,n}.$

\begin{proposition}
\label{prop:linear} The space $u_{m,n}\in H_{0}^{2}(I;\mathbb{C}^{3}%
\mathbb{)}$ has a representation as
\[
u_{m,n}(z)=\sum_{l\in\mathbb{Z}^{+}}u_{m,n,l}e_{m,n,l}(z),
\]
where $u_{m,n,l}\in\mathbb{C}^{3}$ and $e_{m,n,l}(z)$ is the eigenfunction of
$L$ with eigenvalue $\lambda_{m,n,l}$. Moreover, the eigenfunction
$\lambda_{m,n,l}(\tau)$ crosses zero only at the critical field%
\[
\tau_{m,n,l}=\varepsilon p^{2}+\frac{1}{\varepsilon}(l/p)^{2}\text{,}%
\]
where
\[
p^{2}=(am)^{2}+(bn)^{2}+l^{2}.
\]
In addition, the eigenfunctions satisfy the relations
\[
e_{m,n,l}(z)=R_{x}e_{-m,n,l}(z)=R_{y}e_{m,-n,l}(z)=(-1)^{l+1}e_{m,n,l}%
(-z)\text{.}%
\]

\end{proposition}

\begin{proof}
From the associated Euler-Lagrange equations, we have that the second
variation at the flat layered state applied to $u=(w,v)$ is%
\[
\delta^{2}\tilde{F}(0)u=\frac{1}{2}\int_{\Omega}\frac{1}{\varepsilon
}\left\vert \nabla v-(w,0)\right\vert ^{2}+\varepsilon\left\vert \nabla
w\right\vert ^{2}-\tau\left\vert w\right\vert ^{2}=\left\langle
Lu,u\right\rangle _{L^{2}}\text{,}%
\]
then the operator $L$ is given by
\begin{equation}
L(w,v)=(-\varepsilon\Delta w-\frac{1}{\varepsilon}\nabla_{\parallel}v+\frac
{1}{\varepsilon}w-\tau w,-\frac{1}{\varepsilon}\Delta v+\frac{1}{\varepsilon
}\nabla_{\parallel}\cdot w)\text{,}%
\end{equation}
where $\nabla_{\parallel}=(\partial_{x},\partial_{y})$. Then the eigenfunctions
$e_{m,n,l}(z)$ and eigenvalues $\lambda_{m,n,l}$ are such that
\begin{equation}
L\left(  e_{m,n,l}(z)e^{i(amx+bny)}\right)  =\lambda_{m,n,l}e_{m,n,l}%
(z)e^{i(amx+bny)}\text{.} \label{ei}%
\end{equation}

From here after we omit the index $m,n,l$ to simplify notation. We set
$e_{m,n,l}(z)=(w_{1},w_{2},v)$,%
\[
\alpha=am\text{ and }\beta=bn.
\]

We have that the condition (\ref{ei}) is equivalent to
\begin{align}
-\varepsilon w_{1}^{^{\prime\prime}}+\varepsilon(\alpha^{2}+\beta^{2}%
)w_{1}-i\alpha v/\varepsilon+(\varepsilon^{-1}-\tau)w_{1}  &  =\lambda
w_{1},\nonumber\label{euler}\\
-\varepsilon w_{2}^{^{\prime\prime}}+\varepsilon(\alpha^{2}+\beta^{2}%
)w_{2}-i\beta v/\varepsilon+(\varepsilon^{-1}-\tau)w_{2}  &  =\lambda w_{2},\\
-v^{^{\prime\prime}}/\varepsilon+(\alpha^{2}+\beta^{2})v/\varepsilon+i(\alpha
w_{1}+\beta w_{2})/\varepsilon &  =\lambda v.\nonumber
\end{align}
Writing
\begin{equation}
f=\alpha w_{1}+\beta w_{2},\quad g=\alpha w_{1}-\beta w_{2},\quad
\text{and}\quad h=iv/\varepsilon, \label{define-fg}%
\end{equation}
the system (\ref{euler}) becomes
\begin{align}
-\varepsilon f^{^{\prime\prime}}+[\varepsilon(\alpha^{2}+\beta^{2}%
)+(\varepsilon^{-1}-\tau)]f-(\alpha^{2}+\beta^{2})h  &  =\lambda
f,\label{equation-f}\\
-\varepsilon g^{^{\prime\prime}}+[\varepsilon(\alpha^{2}+\beta^{2}%
)+(\varepsilon^{-1}-\tau)]g-(\alpha^{2}-\beta^{2})h  &  =\lambda
g,\label{equation-g}\\
-\varepsilon h^{^{\prime\prime}}+\varepsilon(\alpha^{2}+\beta^{2})h-f  &
=\lambda\varepsilon^{2}h. \label{equation-h}%
\end{align}
with boundary conditions $(f,g,h)(\pm\pi/2)=0$.

We consider the linear operator
\[
Lh=-\varepsilon h^{^{\prime\prime}}+[\varepsilon(\alpha^{2}+\beta
^{2})-\varepsilon^{2}\lambda]h,
\]
defined in $\mathcal{D}(L)=\{h\in H^{2}(I;\mathbb{R})|h(\pm\pi/2)=0\}$. Let
the inverse be $T=L^{-1}$, it is well known that $T$ is compact. Then, the
system (\ref{equation-f}) and (\ref{equation-h}) can be written as $h=Tf$ and
\begin{equation}
f=\left(  \lambda(1-\varepsilon^{2})+\tau-\varepsilon^{-1}\right)
Tf+(\alpha^{2}+\beta^{2})T^{2}f=p(T)f, \label{T-theta}%
\end{equation}
where we have defined the polynomial
\[
p(z)=\left(  \lambda(1-\varepsilon^{2})+\tau-\varepsilon^{-1}\right)
z+(\alpha^{2}+\beta^{2})z^{2}.
\]

Let $\kappa$ be an eigenvalue of $T$, i.e., $Tf=\kappa f$, then $f$ satisfies
the equation $Lf=\kappa^{-1}f$. From the Dirichlet boundary conditions we see
that the solutions are
\[
f_{l}(x)=\left\{
\begin{array}
[c]{c}%
\sin(lz)\text{ for }l\in2\mathbb{Z}^{+}\\
\cos(lz)\text{ for }l\in2\mathbb{Z}^{+}\mathbb{-}1
\end{array}
\right.  \text{.}%
\]
Therefore, the eigenvalues $\kappa$ need to satisfy
\begin{equation}
\kappa^{-1}=\varepsilon(\alpha^{2}+\beta^{2}+l^{2})-\varepsilon^{2}\lambda.
\label{R1}%
\end{equation}

Since $T$ is compact, from (\ref{T-theta}) and the fact that $\sigma
(p(T))=p(\sigma(T))$ where $\sigma(T)$ is a spectrum of $T$ and $p$ is a
polynomial, we have $1\in p(\sigma(T)).$ This can be written as
\begin{equation}
\kappa^{-2}=\left(  \lambda(1-\varepsilon^{2})+\tau-\varepsilon^{-1}\right)
\kappa^{-1}+(\alpha^{2}+\beta^{2}). \label{R2}%
\end{equation}

Therefore eigenvalues are defined by the relations (\ref{R1}) and (\ref{R2}),
using
\[
p^{2}=\alpha^{2}+\beta^{2}+l^{2}%
\]
then%
\[
\left[  \varepsilon p^{2}-\varepsilon^{2}\lambda\right]  ^{2}=\left(
\lambda(1-\varepsilon^{2})+\tau-\varepsilon^{-1}\right)  \left[  \varepsilon
p^{2}-\varepsilon^{2}\lambda\right]  +(\alpha^{2}+\beta^{2})
\]

The previous relation is equivalent to%
\[
\left(  p^{2}-\lambda\varepsilon\right)  (\varepsilon^{2}p^{2}-\lambda
\varepsilon-\tau\varepsilon)+l^{2}-\lambda\varepsilon=0\text{.}%
\]
Now $\lambda=0$ is when
\[
\tau_{0}=\varepsilon p^{2}+\frac{1}{\varepsilon}\frac{l^{2}}{p^{2}}.
\]
Deriving respect $\tau$ we have for $\lambda^{\prime}(\tau)$ that%
\[
\left(  -\lambda^{\prime}\varepsilon\right)  (\varepsilon^{2}p^{2}%
-\lambda\varepsilon-\tau\varepsilon)+\left(  p^{2}-\lambda\varepsilon\right)
(-\lambda^{\prime}\varepsilon-\varepsilon)-\lambda^{\prime}\varepsilon=0
\]
evaluating at the critical value $\tau_{0}$ we have $\lambda=0$, then%
\[
\lambda^{\prime}(\tau_{0})=\frac{p^{2}}{(l/p)^{2}-p^{2}-1}\neq0
\]
as $p\neq0$. Therefore, the eigenvalue $\lambda$ crosses zero at $\tau_{0}$.

The solution to (\ref{equation-g}) is given by
\begin{equation}
g=\frac{\alpha^{2}-\beta^{2}}{\alpha^{2}+\beta^{2}}f, \label{g}%
\end{equation}
which leads to the eigenfunction
\[
(w_{1},w_{2},v)=\left(  \frac{am}{\alpha^{2}+\beta^{2}},\frac{bn}{\alpha
^{2}+\beta^{2}},\frac{-i}{p^{2}-\lambda\varepsilon}\right)  f_{l}(z).
\]

Moreover, we would like to choose the eigenfunctions such that the condition
$u_{-m,-n}=\bar{u}_{m,n}$ becomes $u_{-m,-n,l}=\bar{u}_{m,n,l}$. Then, the
eigenfunction $e_{m,n,l}(z)$ such that $e_{-m,-n,l}(z)=\bar{e}_{m,n,l}(z)$ are%
\begin{equation}
e_{m,n,l}(z)=\left(  \frac{ami}{\alpha^{2}+\beta^{2}},\frac{bni}{\alpha
^{2}+\beta^{2}},\frac{1}{p^{2}-\lambda\varepsilon}\right)  f_{l}(z).
\label{u_v}%
\end{equation}
Thus the eigenvalue problem (\ref{euler}) has eigenvalues $\lambda$ which is
zero at the critical value $\tau$ and corresponding eigenfunctions
$e_{m,n,l}(z)$.
\end{proof}

We have considered the domain with sides $\Omega$ in order to simplify the
spectrum, as it appears in terms of $(am)^{2}$, $(bn)^{2}$ and $l^{2}$.We conclude:

\begin{proposition}
For all $u\in\mathcal{H}$ we have
\[
u=\sum_{m,n\in\mathbb{Z}}\sum_{l\in\mathbb{Z}^{+}}u_{m,n,l}e_{m,n,l}%
(z)e^{i(amx+bny)}\text{,}%
\]
and
\[
Lu=\sum_{m,n\in\mathbb{Z}}\sum_{l\in\mathbb{Z}^{+}}\lambda_{m,n,l}%
u_{m,n,l}e_{m,n,l}(z)e^{i(amx+bny)}\text{.}%
\]

\end{proposition}

\begin{remark}
\label{remark:critical-field}In order to find the critical value of $\tau$, we
minimize (\ref{LinearFunctional}) over functions satisfying $\int_{\Omega
}(n_{1}^{2}+n_{2}^{2})=1$. Therefore, one can see that the first critical
field is the minimum of the $\tau_{m,n,l}$%
\begin{equation}
\tau_{c}=\min_{m,n,l}\tau_{m,n,l}. \label{threshold1}%
\end{equation}
Following the proof of Proposition \ref{prop:linear}, one can see that
$\tau_{m,n,l}$ can be achieved from the system (\ref{equation-f}%
)-(\ref{equation-h}) with $\lambda=0$. Noticing that (\ref{equation-f}) and
(\ref{equation-h}) are decoupled from the system, we rewrite them, with
$\delta=\varepsilon(\alpha^{2}+\beta^{2})$ and $\psi=\delta h$,
\begin{align*}
&  -\varepsilon f^{^{\prime\prime}}+\Big(\delta+\frac{1}{\varepsilon
}\Big)f-\frac{1}{\varepsilon}\psi=\tau f,\\
& \\
5ex]  &  -\psi^{^{\prime\prime}}+\frac{\delta}{\varepsilon}\psi=\frac{\delta
}{\varepsilon}f.
\end{align*}
Thus this system determines $\tau_{c}$ in (\ref{threshold1}) and it has been
studied for layer undulations in two dimensions \cite{G-J2} with the magnetic
field applied in the $x$ direction. Thus the first eigenvalue obtained in
\cite{G-J2} for two dimensions also provides the estimate of the critical
field for our three dimensional case. In particular, Theorem 4 of \cite{G-J2}
proves that the first eigenvalue is estimated as $\pi$ when the interval for
$z$ is $(-1,1)$. Since the scaling used in this paper gives the interval
$(-\pi/2,\pi/2)$ for $z$, it is easy to see that the first eigenvalue is
estimated
\begin{equation}
\tau_{c}\sim2. \label{first-eigenvalue}%
\end{equation}

\end{remark}

\section{Bifurcation}

\label{bifurcation}

{As anticipated, the first step is to identify the irreducible representations
associated to $G.$ Due to Schur's lemma, all eigenvalues }$\lambda$ of $L$
corresponding to this representation are the same, and then, if one does not
consider the symmetries the resonances make the proof of the bifurcation
difficult or impossible.

In the second step, the strategy is to identify the {isotropy groups that have
a fixed point space of dimension one in the irreducible representation, and
then, using the implicit function theorem in the restricted point space of this
isotropy groups, we can obtain simple bifurcation, that is without resonances.
}However, we do not claim to find them all, although the two we find are
enough to characterize the loss of stability of the uniform state to a
preferred branch while at the same time shows the emergence of other
interesting solutions satisfying different symmetries, as soon as one crosses
the critical field $\tau_{c}.$

{The symmetries of the bifurcation solutions are given by the isotropy groups
with these properties. The symmetries of the solutions are presented in the
third part of this section, and the bifurcation theorem in
the fourth subsection. }

\subsection{Irreducible representations}

Now, we proceed to find the irreducible representations. It suffices to
characterize the action on the elements of the Fourier basis.

\begin{proposition}
For fixed $m,n,l>0$, the action on $(u_{1},u_{2})=(u_{m,n,l},u_{m,-n,l})$ is
given by
\begin{align*}
\rho(\varphi,\theta)(u_{1},u_{2})  &  =e^{im\varphi}(e^{in\theta}%
u_{1},e^{-in\theta}u_{2})\text{,}\\
\rho(\kappa_{x})(u_{1},u_{2})  &  =(\bar{u}_{2},\bar{u}_{1})\text{, }\\
\rho(\kappa_{y})(u_{1},u_{2})  &  =(u_{2},u_{1})\text{,}\\
\rho(\kappa_{z})(u_{1},u_{2})  &  =(-1)^{l}(u_{1},u_{2}).
\end{align*}

\end{proposition}

\begin{proof}
The action of $G$ in each Fourier component is given by%
\[
\rho(\varphi,\theta)u(x,y)=u(x+\varphi/a,y+\theta/b)=\sum_{m,n\in\mathbb{Z}%
}\left(  e^{im\varphi}e^{in\theta}u_{m,n}\right)  e^{i(amx+bny)}\text{.}%
\]
From this we obtain
\[
\rho(\varphi,\theta)u_{m,n}(z)=e^{im\varphi}e^{in\theta}u_{m,n}(z).
\]

In a similar way, we have that%
\begin{align*}
\rho(\kappa_{x})u_{m,n}  &  =R_{x}u_{-m,n}\text{, }\\
\rho(\kappa_{y})u_{m,n}  &  =R_{y}u_{m,-n}\text{,}\\
\text{ }\rho(\kappa_{z})u_{m,n}(z)  &  =-u_{m,n}(-z).
\end{align*}

Moreover, since $R_{x}e_{-m,n,l}(z)=e_{m,n,l}(z)$, then%
\[
R_{x}u_{-m,n}=\sum_{l\in\mathbb{Z}}u_{-m,n,l}R_{x}e_{-m,n,l}(z)=\sum
_{l\in\mathbb{Z}}u_{-m,n,l}e_{m,n,l}(z)\text{,}%
\]
and then the reflection $\kappa_{x}$ acts according to
\[
\rho(\kappa_{x})u_{m,n,l}=u_{-m,n,l}=\bar{u}_{m,-n,l}\text{.}%
\]
In a similar way, using the fact that $R_{y}e_{m,-n,l}(z)=e_{m,n,l}(z)$, one
can prove that
\[
\rho(\kappa_{y})u_{m,n,l}=u_{m,-n,l}.
\]

Furthermore, since $-e_{m,n,l}(-z)=(-1)^{l}e_{m,n,l}(z)$, then $\ $%
\[
\rho(\kappa_{z})u_{m,n}(z)=-u_{m,n}(-z)=\sum_{l\in\mathbb{Z}^{+}}%
(-1)^{l}u_{m,n,-l}e_{m,n,l}(z)\text{.}%
\]

\end{proof}

We conclude that the irreducible representations correspond to $(u_{1}%
,u_{2})\in\mathbb{C}^{2}$ with the previous action when $m,n\neq0$. In the
case that $m$ or $n$ are zero, the irreducible representations are
$\mathbb{C}$, however we do not analyze this case because this produces a
bifurcation of stationary solutions in $x$ or $y$, which may be analyzed
directly \cite{G-J2}.

\subsection{Isotropy groups}

For each representation there are at least two kind of isotropy groups that
have fixed point spaces of dimension one. Bear in mind that, as mentioned
earlier, in this section we do not claim to find all of them, but at least
two. Hereafter, we denote
\[
(u_{1},u_{2})=(u_{m,n,l},u_{m,-n,l})\in\mathbb{C}^{2}%
\]
the irreducible representation {for a fixed }$m,n,l$.

The irreducible representation for $l$ even is fixed by the element
$\kappa_{z}$. Moreover, {the point $(u_{1},u_{2})=(r,0)$ for $r \in\mathbb{R}$
has an isotropy group given by}%
\[
O(2)\times\mathbb{Z}_{2}=\left\langle (\varphi/m,-\varphi/n),\kappa_{x}%
\kappa_{y},\kappa_{z}\right\rangle \text{,}%
\]
where { $\langle g_{1},g_{2},\ldots,g_{n}\rangle$ denotes, as usual, the
subgroup generated by the elements $g_{1},\ldots,g_{n}$ of }${G}${. The other
isotropy groups }corresponds to $(u_{1},u_{2})=(r,r)$ and is given by%
\[
D\times\mathbb{Z}_{2}=\left\langle (\pi/m,-\pi/n),\kappa_{x},\kappa_{y}%
,\kappa_{z}\right\rangle \text{.}%
\]

For $l$ odd one isotropy group corresponds to $(u_{1},u_{2})=(ir,0)$ and is
given by%
\[
\tilde{O}(2)=\left\langle (\varphi/m,-\varphi/n),\kappa_{x}\kappa_{y}%
\kappa_{z}\right\rangle \text{.}%
\]
The second isotropy group corresponds to $(u_{1},u_{2})=(r,-r)$ is given by%
\[
\tilde{D}=\left\langle (\pi/m,-\pi/n),\kappa_{z}\kappa_{x},\kappa_{z}%
\kappa_{y}\right\rangle
\]

It is not hard to prove that the fixed point spaces of these groups have real
dimension equal to one, where the fixed point spaces are generated by
$(1,0);(1,1);(i,0)$ and $(1,-1)$ respectively.

\subsection{Symmetries}

{The previous isotropy groups are relevant because they give the symmetries of
the bifurcating solutions. Before proving the existence of these branches, we
want to present the symmetries of the bifurcating solutions, i.e. the fixed
point spaces of the previous groups. }

Since the elements $(\varphi/m_{0},-\varphi/n_{0})$, $\kappa_{x}\kappa_{y}$
and $\kappa_{z}$ generate $O(2)\times\mathbb{Z}_{2}$, solutions with this
isotropy group must satisfy
\begin{align}
u(x,y,z)  &  =R_{x}R_{y}u(-x,-y,z)=-u(x,y,-z)\label{Oa}\\
&  =u(x+\varphi/\alpha,y-\varphi/\beta,z)\text{,}\nonumber
\end{align}
where $\alpha=am_{0}$ and $\beta=bn_{0}$.

From the first symmetry, we have that
\[
u=\sum_{m,n\in\mathbb{Z}}e^{i((m/m_{0})-(n/n_{0}))\varphi}u_{m,n}%
e^{i(amx+bny)}\text{.}%
\]
Then $u_{m,n}=0$ unless $(m/m_{0})-(n/n_{0})=0$ or $m/n=m_{0}/n_{0}$, then
$m=jm_{0}$ and $n=jn_{0}$ for $j\in\mathbb{Z}$ and%

\[
u=\sum_{j\in\mathbb{Z}}u_{j}(z)e^{ij(am_{0}x+bn_{0}y)}%
\]
with $u_{j}(z)=u_{jm_{0},jn_{0}}(z)$ with
\[
u_{j}(z)=\bar{u}_{-j}(z)=R_{x}R_{y}\bar{u}_{j}(z)=-u_{j}(-z).
\]

Since the elements $(\varphi/m_{0},-\varphi/n_{0})$, $\kappa_{x}\kappa
_{y}\kappa_{z}$ generate $\tilde{O}(2)$, they must satisfy
\begin{equation}
u(x,y,z)=u(x+\varphi/\alpha,y-\varphi/\beta,z)=R_{z}u(-x,-y,-z)\text{,}
\label{Ob}%
\end{equation}
where $\ $ $R_{z}=-R_{x}R_{y}$. Therefore, $u$ is as before but with
$u_{j}(z)=u_{jm_{0},jn_{0}}(z)$ such that
\[
u_{j}(z)=\bar{u}_{-j}(z)=R_{z}\bar{u}_{j}(-z).
\]

The isotropy group $D\times\mathbb{Z}_{2}$ has the generators $\kappa_{x}$,
$\kappa_{y}$, $\kappa_{z}$ and $(\pi/m_{0},-\pi/n_{0})$, then solutions with
this isotropy group must satisfy
\begin{align}
u(x,y,z)  &  =R_{x}u(-x,y,z)=R_{y}u(x,-y,z)=-u(x,y-z)\label{Da}\\
&  =u(x+\pi/\alpha,y-\pi/\beta,z)\text{,}\nonumber
\end{align}
where $\alpha=am_{0}$, $\beta=bn_{0}$. In this case%
\[
u=\sum_{m,n\in\mathbb{Z}}e^{i((m/m_{0})-(n/n_{0}))\pi}u_{m,n}(z)e^{i(amx+bny)}%
\text{, }%
\]
then $u_{m,n}=0$ unless $(m/m_{0})-(n/n_{0})\in2\mathbb{Z}$. Moreover, we have
that
\[
u_{m,n}(z)=\bar{u}_{-m,-n}(z)=R_{x}u_{-m,n}(z)=R_{y}u_{m,-n}(z)=-u_{m,n}%
(-z)\text{.}%
\]

The group $\tilde{D}$ has generators $\kappa_{z}\kappa_{x}$, $\kappa_{z}%
\kappa_{y}$ and $(\pi/m_{0},-\pi/n_{0})$. Then, solutions with this isotropy
group $\tilde{D}$ must satisfy
\begin{align}
u(x,y,z)  &  =-R_{x}u(-x,y,-z)=-R_{y}u(x,-y,-z)\label{Db}\\
&  =u(x+\pi/\alpha,y-\pi/\beta,z).\nonumber
\end{align}
As before, we have that $u_{m,n}=0$ unless $(m/m_{0})-(n/n_{0})\in2\mathbb{Z}$
and%
\[
u_{m,n}(z)=\bar{u}_{-m,-n}(z)=-R_{x}u_{-m,n}(-z)=-R_{y}u_{m,-n}(-z)\text{.}%
\]

\subsection{Bifurcation theorem}

\begin{theorem}
For each fixed $m_{0},n_{0},l_{0}\geq1$, the uniformly layered state $u=0$ has
two bifurcations of critical points of \eqref{energy} starting from the critical field $\tau_{m_{0},n_{0}%
,l_{0}}$ in $\mathcal{H},$ except for a finite union of hypersurfaces in the set of parameters
$(a,b,\varepsilon)$. One of the bifurcations has the symmetries of $O(2)\times
\mathbb{Z}_{2}$ and the other one of $D\times\mathbb{Z}_{2}$ for $l_{0}$ even, and $\tilde
{O}(2)$ and $\tilde{D}$ for $l_{0}$ odd.
\end{theorem}

\begin{proof} We appeal to Theorem \ref{T5}.
Since the eigenvalue $\lambda_{m,n,l}$ is zero only at $\tau_{m,n,l}$, and
$\tau_{m,n,l}\rightarrow\infty$ as $m,n,l\rightarrow\infty$, then
$\lambda_{m,n,l}(\tau_{m_{0},n_{0},l_{0}})$ is different from zero except for
a finite number of $(m,n,l)\in\mathbb{Z}^{3}$. Moreover, the condition
$\tau_{m,n,l}-\tau_{m_{0},n_{0},l_{0}}=0$ defines a hypersurface
in the space of parameters $(a,b,\varepsilon).$ On the other hand $\tau_{m,n,l}-\tau
_{m_{0},n_{0},l_{0}}$ is identically zero only if $l^{2}=l_{0}^{2}$ and $m_{0}%
^{2}=m^{2}$ and $n_{0}^{2}=n^{2}.$

Therefore, there are no resonances with others eigenvalues $\lambda_{m_{0}%
,n_{0},l_{0}}$ except for the set of parameters $(a,b,\varepsilon)$ in this finite number of hypersurfaces. Thus, we conclude that the
only eigenvalue $\lambda_{m,n,l}$ close to zero at $\tau_{m_{0},n_{0},l_{0}}$
is $\lambda_{m_{0},n_{0},l_{0}}$. Moreover, {the kernel of }${L}${ is given by the space corresponding to the
irreducible representation $(u_{m_{0},n_{0},l_{0}},u_{m_{0},-n_{0},l_{0}}%
)\in\mathbb{C}^{2}$, i.e. the eigenvalue $\lambda_{m_{0},n_{0},l_{0}}$ has
real dimension equal to four. We then take advantage of the actions
of the isotropy groups that have a fixed point space of real dimension equal
to one in $\mathbb{C}^{2}$}; in either one of the fixed point spaces ( $O(2)\times
\mathbb{Z}_{2}$ and $D\times\mathbb{Z}_{2}$ for $l_{0}$ even, and $\tilde
{O}(2)$ and $\tilde{D}$ for $l_{0}$ odd) of these groups the
kernel has real dimension equal to one. Therefore, the restricted map $f$ to
the fixed point space of the isotropy group. For this restricted map only one
eigenvalue crosses zero, which then yields the desired result as an application of Theorem
\ref{BifTheo}.
\end{proof}

Unfortunately, there are resonances in the hypersurface of parameters $a=b$,
i.e. we have that $\lambda_{m_{0},n_{0},l}=\lambda_{n_{0},m_{0},l}$. In this
case, the previous theorem does not apply, but we can improve the results of
the previous theorem for this particular case.

\begin{proposition}
\label{propab} In the case that $a=b$, there is always a bifurcation of
solutions from $\tau_{m,n,l}$ for the groups $O(2)\times\mathbb{Z}_{2}$ and
$\tilde{O}(2)$, as in the previous theorem{, and for the groups}
$D\times\mathbb{Z}_{2}$ and $\tilde{D}$ if in addition the only solutions of
\[
m^{2}+n^{2}=m_{0}^{2}+n_{0}^{2}%
\]
are $(m,n)$ is equal to $(m_{0},n_{0})$ or $(n_{0},m_{0})$.
\end{proposition}

\begin{proof}
In this case, we can use the same arguments as in the previous theorem, we are
only left to verify that there are no resonances, that is, we only need to
consider the cases when $\tau_{m,n,l}-\tau_{m_{0},n_{0},l_{0}}$ is identically
zero(viewing both expressions as functions of $\varepsilon$). In the fixed point space $O(2)\times\mathbb{Z}_{2}$ and $\tilde{O}(2)$
other resonances cannot appear, since the fixed point space has functions $u$
with components $u_{m,n}$ equal to zero unless $(m,n)=j(m_{0},n_{0})$ for
$j\in\mathbb{Z}$, and then, we have that $\lambda_{m,n,l}\neq\lambda
_{m_{0},n_{0},l}$ for $n>0$ unless $(m,n)=(m_{0},n_{0})$.

In the case of the groups $D\times\mathbb{Z}_{2}$ and $\tilde{D}$, we have
from the assumptions that $m^{2}+n^{2}=m_{0}^{2}+n_{0}^{2}$ only for $(m,n)$
equal to $(m_{0},n_{0})$ and $(n_{0},m_{0})$, then we have the double
eigenvalue $\lambda_{m_{0},n_{0},l}=\lambda_{n_{0},m_{0},l}$. Actually, this
double eigenvalue is a consequence of the invariance of $F$ by another action
$\kappa\in\mathbb{Z}_{2}$ given by $\rho(\kappa)u(x,y)=u(y,x)$, due to
$a=b$. Using this extra symmetry, it is not hard to see that in the fixed
point space of $\kappa$ and the group $D\times\mathbb{Z}_{2}$ or $\tilde{D}$,
no additional resonances exist by assumptions. Therefore, the desired conclusion follows again in this case thanks to Theorem \ref{T5}.
\end{proof}

\begin{remark}
The assumptions in the case $a=b$ are true for many $(m_{0},n_{0})$, including
the case $m_{0}=n_{0}=1$.
\end{remark}

\begin{remark}
As all the eigenvalues $\lambda_{m_{0},n_{0},l}$ cross in the same direction,
then $\tau_{m_{0},n_{0},l_{0}}$ is always a bifurcation point regardless of the
multiplicity. But the previous theorems allows us to obtain specific
information about different branches, the symmetries of the branches, the
local continuity of the branches, and local estimates given in the following section.
\end{remark}

\subsection{Local estimates} \label{localestimates}

We complete the bifurcation analysis by providing the asymptotic portrait of the branches found in Theorem \ref{Tmain}. As mentioned in the introduction, these properties follow from standard bifurcation arguments \cite{CrandallRab}.

Since the bifurcation is simple, we can estimate the local bifurcating
branches in the cases $a\neq b$.
\vskip.1in
{\bf 1. The branches with $O(2)\times\mathbb{Z}_{2}$ symmetry}
\vskip.1in
The bifurcation with group $O(2)\times\mathbb{Z}_{2}$ has leading term
$(u_{m_{0},n_{0},l_{0}},u_{m_{0},-n_{0},l_{0}})$ $=(r/2,0)$, from where we gather

\vskip.1in
\fbox{
\addtolength{\linewidth}{-2\fboxsep}%
\addtolength{\linewidth}{-2\fboxrule}%
\begin{minipage}{\linewidth}
\[
u=r\left(  e_{m_{0},n_{0},l_{0}}(z)e^{i(am_{0}x+bn_{0}y)}+e_{-m_{0}%
,-n_{0},l_{0}}(z)e^{-i(am_{0}x+bn_{0}y)}\right)  +O(r^{2})
\]
\end{minipage}
}
\vskip.1in

where $O(r^{2})$ is a function in the fixed point space of $O(2)\times
\mathbb{Z}_{2}$ of order $r^{2}$. Moreover, using (\ref{u_v}) we conclude that%
\vskip.1in
\fbox{
\addtolength{\linewidth}{-2\fboxsep}%
\addtolength{\linewidth}{-2\fboxrule}%
\begin{minipage}{\linewidth}
\begin{eqnarray*}
u=r\left(  \frac{-am_{0}\sin(am_{0}x+bn_{0}y)}{\alpha^{2}+\beta^{2}}%
,\frac{-bn_{0}\sin(am_{0}x+bn_{0}y)}{\alpha^{2}+\beta^{2}},\frac{\cos
(am_{0}x+bn_{0}y)}{p^{2}-\lambda\varepsilon}\right) &&  \\
\times\sin(l_{0}z)+O(r^{2}%
)\text{.}
\end{eqnarray*}
\end{minipage}
}
\vskip.1in

The previous solutions satisfy the corresponding symmetry (\ref{Oa}). Also,
notice that $u$ is the parameterization of $(\psi(u),n(u))$ given in
(\ref{cif}), then one can approximately characterize these functions in the local
branches%
\vskip.1in
\fbox{
\addtolength{\linewidth}{-2\fboxsep}%
\addtolength{\linewidth}{-2\fboxrule}%
\begin{minipage}{\linewidth}
\[\psi =\psi_{0}+i\frac{\psi_{0}}{c\varepsilon}\left(  \frac{\cos
(am_{0}x+bn_{0}y)}{p^{2}-\lambda\varepsilon}\right)  r\sin(l_{0}%
z)+O(r^{2}) \text{,}\]
\begin{eqnarray*}
n   =\left(  \frac{-am_{0}\sin(am_{0}x+bn_{0}y)}{\alpha^{2}+\beta^{2}}%
r\sin(l_{0}z),\frac{-bn_{0}\sin(am_{0}x+bn_{0}y)}{\alpha^{2}+\beta^{2}}%
r\sin(l_{0}z),1\right)&&  \\
+O(r^{2})\text{.}
\end{eqnarray*}
\end{minipage}
}
\vskip.1in

\vskip.1in
{\bf 1. The branches with $D\times\mathbb{Z}_{2}$ symmetry}
\vskip.1in

For the group $D\times\mathbb{Z}_{2}$ we have that $(u_{m_{0},n_{0},l_{0}%
},u_{m_{0},-n_{0},l_{0}})=(r/2,r/2)$, thus%
\vskip.1in
\fbox{
\addtolength{\linewidth}{-2\fboxsep}%
\addtolength{\linewidth}{-2\fboxrule}%
\begin{minipage}{\linewidth}
\[
u=r\operatorname{Re}(e_{m_{0},n_{0},l_{0}}(z)e^{i(am_{0}x+bn_{0}y)}%
+e_{m_{0},-n_{0},l_{0}}(z)e^{i(am_{0}x-bn_{0}y)})+O(r^{2})\text{,}%
\]
\end{minipage}
}
\vskip.1in
and as before we have the local solutions that satisfy the symmetry
(\ref{Da}) given by%
\vskip.1in
\fbox{
\addtolength{\linewidth}{-2\fboxsep}%
\addtolength{\linewidth}{-2\fboxrule}%
\begin{minipage}{\linewidth}
\begin{eqnarray*}
u=\sum_{n=\pm n_{0}}r\Big(  \frac{-am_{0}\sin(am_{0}x+bny)}{\alpha^{2}%
+\beta^{2}},\frac{-bn\sin(am_{0}x+bny)}{\alpha^{2}+\beta^{2}},\frac
{\cos(am_{0}x+bny)}{p^{2}-\lambda\varepsilon}\Big) &&  \\
\times \sin(l_{0}%
z)+O(r^{2})\text{.}
\end{eqnarray*}
\end{minipage}
}
\vskip.1in

For the group $\tilde{O}(2)$ we have that $(u_{m,n,l},u_{m,-n,l})=(-ir/2,0)$,
and so
\vskip.1in
\fbox{
\addtolength{\linewidth}{-2\fboxsep}%
\addtolength{\linewidth}{-2\fboxrule}%
\begin{minipage}{\linewidth}
\[
u=r\operatorname{Im}(e_{m_{0},n_{0},l_{0}}(z)e^{i(am_{0}x+bn_{0}y)}%
)+O(r^{2})\text{.}%
\]
\end{minipage}
}
\vskip.1in
Using (\ref{u_v}) we have that the solution that satisfy (\ref{Ob}) locally
is
\vskip.1in
\fbox{
\addtolength{\linewidth}{-2\fboxsep}%
\addtolength{\linewidth}{-2\fboxrule}%
\begin{minipage}{\linewidth}
\begin{eqnarray*}
u=r\left(  \frac{am_{0}\cos(am_{0}x+bn_{0}y)}{\alpha^{2}+\beta^{2}}%
,\frac{bn_{0}\cos(am_{0}x+bn_{0}y)}{\alpha^{2}+\beta^{2}},\frac{-\sin
(am_{0}x+bn_{0}y)}{p^{2}-\lambda\varepsilon}\right)&& \\
\times\cos(l_{0}z)+O(r^{2}%
)\text{.}
\end{eqnarray*}
\end{minipage}
}
\vskip.1in

For the group $\tilde{D}$ we have that $(u_{m,n,l},u_{m,-n,l})=(r/2,-r/2)$,%
\vskip.1in
\fbox{
\addtolength{\linewidth}{-2\fboxsep}%
\addtolength{\linewidth}{-2\fboxrule}%
\begin{minipage}{\linewidth}
\[
u=r\operatorname{Re}(e_{l_{0},m_{0},n_{0}}(z)e^{iam_{0}x}e^{ibn_{0}y}%
-e_{l_{0},m_{0},-n_{0}}(z)e^{iam_{0}x}e^{-ibn_{0}y})+O(r^{2})\text{,}%
\]
\end{minipage}
}
\vskip.1in
where a similar expression with the symmetry (\ref{Db}) may be obtained.

\section{2D planar configuration}

\label{planar}

\label{section:planar} In this section, we study the de Gennes free energy in
the cross section of the rectangular box while assuming the constant smectic
order parameter and find the characteristics of the minimizers well above the
critical field. We also assume a sufficiently strong magnetic field which
forces a planar configuration. Thus by taking $\psi= e^{i \varphi/ c
\varepsilon} $ and $\tau= 1/\varepsilon^{\delta}$ with $1< \delta<2$, the
energy (\ref{energy}) becomes
\begin{equation}
\label{energy-planar}F_{\varepsilon} (\varphi, n) = \int_{\mathbb{T}^{2} }
\left(  \varepsilon|\nabla n|^{2} + \frac1 \varepsilon|\nabla\varphi-
n^{\parallel}|^{2} + \frac{ 1}{ \varepsilon^{\delta}} n_{3}^{2} \right)
\end{equation}
where $\varphi= \varphi(x,y)$, $n = (n^{\parallel}, n_{3}) = (n_{1}, n_{2},
n_{3})$. We assume that $\mathbb{T}^{2} $ is a flat torus identified with
$[0,1)^{2}$. If $(\varphi, n)$ is a minimizer of $F_{\varepsilon}$, then
$\varphi= \varphi_{n}$ satisfies
\begin{equation}
\label{vp-n}\Delta\varphi_{n} = \nabla\cdot n^{\parallel} \quad\mbox{in}
\quad\mathbb{T}^{2} \qquad\mbox{and} \qquad\int_{\mathbb{T}^{2} } \varphi_{n}
= 0.
\end{equation}
Letting $\varphi_{n}$ be a solution to (\ref{vp-n}) and defining
\begin{equation}
\label{energy-planar}G_{\varepsilon}(n) = \int_{\mathbb{T}^{2} } \left(
\varepsilon|\nabla n|^{2} + \frac1 \varepsilon|\nabla\varphi_{n} -
n^{\parallel}|^{2} + \frac{ 1}{ \varepsilon^{\delta}} n_{3}^{2} \right)  ,
\end{equation}
we have
\[
\inf_{(\varphi, n) \in H^{1}(\mathbb{T}^{2} , \mathbb{R}) \times
H^{1}(\mathbb{T}^{2} , \mathbb{S}^{2})} F_{\varepsilon} (\varphi, n) = \inf_{n
\in H^{1}(\mathbb{T}^{2} , \mathbb{S}^{2})} G_{\e}(n).
\]
We therefore study $G_{\varepsilon}$ in (\ref{energy-planar}) with
$\varphi_{n}$ satisfying (\ref{vp-n}) to obtain the configuration of minimizer
of $F_{\varepsilon}$. This formulation was also used in \cite{L-P} where the
weak critical field was achieved.

The functional $G_{\varepsilon}$ in two dimensions is analogous to a free
energy for micromagnetics studied in \cite{Alouges-Serfaty}. To see this, we
let
\[
m = (m^{\parallel}, n_{3}), \quad m^{\parallel} = n_{\perp}^{\parallel} =
(-n_{2}, n_{1}), \qquad\mbox{and} \qquad\nabla^{\perp} \varphi=(-\varphi_{y},
\varphi_{x}).
\]
Then by setting $H = m^{\parallel} -\nabla^{\perp} \varphi$, $G_{\varepsilon}$
becomes
\begin{equation}
\label{energy-m}\int_{\mathbb{T}^{2} } \varepsilon|\nabla m|^{2} + \frac1
\varepsilon|H|^{2} + \frac{ 1}{ \varepsilon^{\delta}} m_{3}^{2},
\end{equation}
where the demagnetizing field $H : \mathbb{T}^{2} \to\mathbb{R}^{2} $ is a
solution to
\begin{equation}
\label{equation-H}\nabla\times H = 0 \qquad\mbox{and} \qquad\nabla\cdot(H -
\mathbf{m} ^{\parallel})= 0 \quad\mbox{ in } \; \mathbb{T}^{2} .
\end{equation}
This energy is not the same as a micromagnetics energy (\ref{energy-m})
studied in \cite{Alouges-Serfaty} due to the presence of a nonlocal term in
(\ref{energy-m}). The following results, however, will follow directly from
\cite{Alouges-Serfaty}. While the limiting $m$ in \cite{Alouges-Serfaty} is a
divergence-free vector field tangent to $\partial\Omega$, the limit $n$ in our
case is a curl-free $\mathbb{S}^{1}$-valued vector field in $\mathbb{T}^{2} $
whose components satisfy zero mass constraint. We restate the following
theorem from \cite{Alouges-Serfaty} in terms of $n$ with the addition of mass
constraint. Let $X \in[0, \frac{\pi}{2}]$ be a geometric half-angle between
$m^{\parallel}_{+}$ and $m^{\parallel}_{-}$ where $m^{\parallel}_{\pm}$ are
traces on each side of the one-dimensional jump set of $m^{\parallel}$. In
terms of $n$, $X$ is simply the half the angle between $n^{\parallel}_{+}$ and
$n^{\parallel}_{-}$.

\begin{proposition}
\label{prop:lower-bound} Let the sequences $\{\varepsilon_{j}\}_{j
\uparrow\infty} \subset(0, \infty)$, and $\{ n_{j}\}_{j \uparrow\infty}\subset
H^{1}(\mathbb{T}^{2} , \mathbb{S}^{2})$ be such that
\[
\varepsilon_{j} \to0 \qquad\mbox { and } \qquad\{ G_{\varepsilon_{j}}%
(n_{j})\}_{j \uparrow\infty} \qquad\mbox{is bounded.}
\]
Then there exist a subsequence $\{ n_{j_{k}} \}$, $n =(n_{1}, n_{2},0) \in
\cap_{p < \infty} L^{p}(\mathbb{T}^{2} , \mathbb{S}^{1})$, and $\varphi\in
\cap_{p < \infty} W^{1,p}(\mathbb{T}^{2} )$ such that
\begin{equation}
\label{limit}\begin{aligned} & n_{j_k} \to n \quad \mbox{ in } \quad \cap_{p < \infty} L^p(\ensuremath{\mathbb{T}^2}), \\ & n = \nabla \ensuremath{\varphi} \quad \mbox{in} \quad \ensuremath{\mathbb{T}^2}, \qquad |\nabla \ensuremath{\varphi} |=1,\\ & \int_{\TT} n_1 = \int_{\TT} n_2 = 0. \end{aligned}
\end{equation}
If, in addition, $n \in BV(\mathbb{T}^{2} , \mathbb{S}^{1})$, then
\begin{equation}
\label{lower-bound}\liminf_{j \to\infty} G_{\varepsilon_{j}}(n_{j}) \geq
\int_{\Sigma_{n}} A(X) \, d\mathcal{H}^{1}%
\end{equation}
where
\begin{align}
&  A(X) = 4|\sin X - X \cos X| \quad\mbox{for} \quad X \in\left[  0,\frac{\pi
}{4}\right]  ,\label{small-angle}\\
&  A(X) = 4|\left(  X-\frac{\pi}{2}\right)  \cos X - \sin X + \sqrt{2}%
|\quad\mbox{for} \quad X \in\left[  \frac{\pi}{4}, \frac{\pi}{2}\right]  .
\label{large-angle}%
\end{align}

\end{proposition}

\begin{proof}
The compactness result may follow directly from Lemma 2.1 in
\cite{Alouges-Serfaty} and the mass constraints of $n_{1}$ and $n_{2}$ in
(\ref{limit}) result from $n = \nabla\varphi$ and periodicity of $\varphi$.
The lower bound inequality (\ref{lower-bound}) follows from Theorem 1 of
\cite{Alouges-Serfaty}.
\end{proof}

Note that constant configurations are not allowed due to the zero mass
constraint in (\ref{limit}) and periodicity of $n$. Therefore, the line
singularity appears, which is also observed in \cite{chevron}. Here, we
consider horizontal or vertical lines for the one-dimensional jump set rather
than diagonal lines, in order to minimize the arc length of the jump set. One
can see that $n$ has at least two internal jumps due to the periodic boundary
condition on $n$. Therefore we should consider two situations, squares with
two jumps on both horizontal and vertical 1d tori (Figure \ref{fig:1d-2d} (a))
and only vertical strips (or horizontal strips) with two parallel 1d-tori
(Figure \ref{fig:1d-2d} (b)). For the latter case, we must have $180^{\circ}$
wall to ensure the mass constraint (\ref{limit}).

\begin{figure}[th]
\centering
\epsfig{file=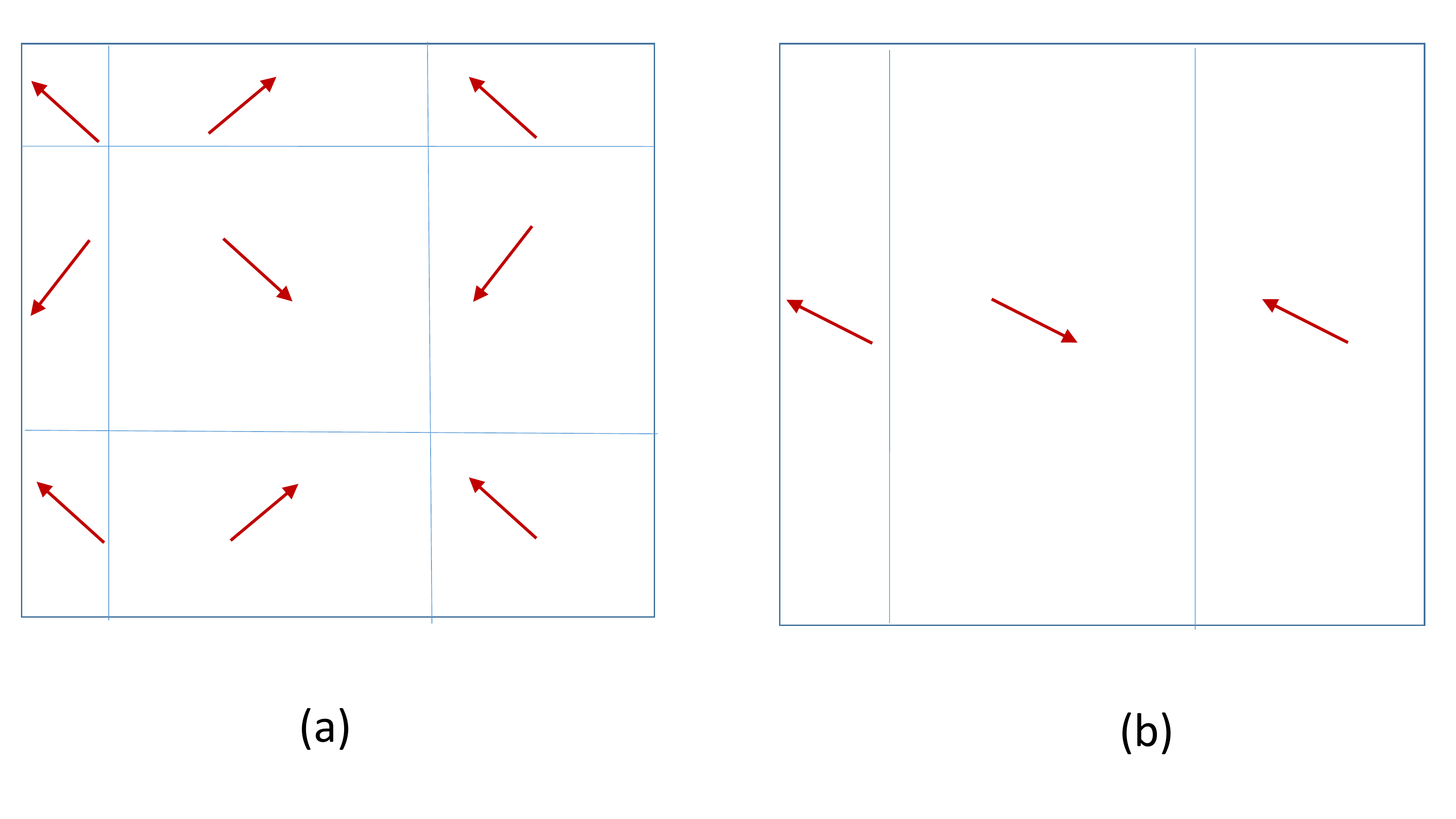,width=0.6\linewidth,clip=} \caption{2d square pattern vs.
1d stripe pattern}%
\label{fig:1d-2d}%
\end{figure}

By applying Proposition \ref{prop:lower-bound}, we would like to compare these two
configurations, 2d square pattern (Figure \ref{fig:1d-2d} (a)) and 1d stripe
configuration (Figure \ref{fig:1d-2d} (b)). For the 1d stripes, the formula
(\ref{large-angle}) with $X= \pi/2$ gives the lower bound
\begin{equation}
\label{energy-1d}G_{\varepsilon}^{1d} \geq8(\sqrt{2}-1) \approx3.3.
\end{equation}
On the other hand, for the 2d square pattern, the formula (\ref{small-angle})
with $X=\pi/4$ gives the better lower bound
\begin{equation}
\label{energy-2d}G_{\varepsilon}^{2d} \geq2 \sqrt{2} (4-\pi) \approx2.4.
\end{equation}
It was proved in \cite{Serfaty-S1} that the lower bound in (\ref{small-angle})
can be achieved by one dimensional structure in the wall for jumps less than
or equal to $\pi/2$. Thus the lower bound in (\ref{energy-2d}) is optimal and
hence we can conclude that 2d square pattern gives the lower energy than 1d
stripe structure.

\section{Numerical Simulations}

\label{sec:Numerics}

\subsection{Complex de Gennes energy}

\label{sec:numerics-complex} We consider the gradient flow (in $L^{2})$ of the
energy (\ref{energy}) and study the behavior of the solutions. The gradient
flow equations are
\begin{equation} \label{GradientFlow}
\begin{aligned}
\frac{\partial\mathbf{n} }{\partial t}  &  = \Pi_{n} \left(  \varepsilon\Delta
n - c \Im{ [\psi(\nabla\psi)^{*} ]} + \tau\left[  (n \cdot\mathbf{e}_{1})
\mathbf{e}_{1} + (n \cdot\mathbf{e}_{2}) \mathbf{e}_{2} \right]  \right)
,\\
\frac{\partial\psi}{\partial t}  &  = c \Delta\psi-2c i n \cdot\nabla\psi- i c
(\nabla\cdot n) \psi- \frac{1}{\varepsilon} \psi+ \frac{g}{\varepsilon}
(1-|\psi|^{2})\psi,
\end{aligned}
\end{equation}
where we have defined, for a given vector $f \in\mathbb{R}^{3}$, the
orthogonal projection onto the plane orthogonal to the vector $n$ as
\[
\Pi_{n}(f) = f-(n \cdot f) n.
\]
This projection appears as a result of the constraint $n \in\mathbb{S}^{2}$.
This method has been used for smectic A liquid crystals \cite{G-J2, G-J3,
G-J4}. As initial condition, we consider a small perturbation from the
undeformed state. More precisely, for all $(x,y,z) \in\Omega$,
\begin{align*}
n (x,y,z,0)  &  = \frac{(\eta u_{1}, \eta u_{2}, 1+ \eta u_{3})}{|(\eta u_{1},
\eta u_{2}, 1+ \eta u_{3}) |},\\
\psi(x,y,z,0)  &  = e^{iz/c \varepsilon} + \eta\psi_{0},
\end{align*}
where a small number $\eta= 0.1$ and $u_{1}, u_{2},u_{3}$ and $\psi_{0}$ are
arbitrarily chosen. We impose strong anchoring condition for the director
field and Dirichlet boundary condition on $\psi$ at the top and the bottom
plates;
\begin{equation}
\label{boundary-n}n(x, y, \pm\frac{\pi}{2}, t) = \mathbf{e}_{3},
\quad\mbox{and} \quad\psi(x, y, \pm\frac{\pi}{2}, t) = e^{iz/c \varepsilon}.
\end{equation}
Periodic boundary conditions are imposed for both $n$ and $\psi$ in the $x$
and $y$ directions.

We use a Fourier spectral discretization in the $x$ and $y$ directions, and
second order finite differences in the $z$ direction. The fast Fourier
transform is computed using the FFTW libraries \cite{F-J}. For the temporal
discretization, we combine a projection method for the variable $n$
\cite{E-W}, with a semi-implicit scheme for $\psi$: Given $(\psi^{k}, n^{k})$,
we solve
\begin{align}
\frac{\mathbf{n} ^{*} -\mathbf{n} ^{k}}{\Delta t}  &  = \varepsilon
\Delta\mathbf{n} ^{*} -c \Im{[\varphi^{k} (\nabla\psi^{k})^{*}]}
+ \tau\left[  (\mathbf{n} ^{k} \cdot\mathbf{e}_{1})
\mathbf{e}_{1} + (\mathbf{n} ^{k} \cdot\mathbf{e}_{2}) \mathbf{e}_{2} \right]
,\label{scheme-n}\\
\mathbf{n} ^{k+1}  &  = \frac{\mathbf{n} ^{*}}{|\mathbf{n} ^{*}|}%
,\label{projection}\\
\frac{\psi^{k+1} - \psi^{k}}{\Delta t}  &  = c \Delta\psi^{k+1} - 2 c i
\mathbf{n} ^{k+1} \cdot\nabla\psi^{k} - i c (\nabla\cdot\mathbf{n} ^{k+1} )
\psi^{k} - \frac{1}{\varepsilon} \psi^{k+1}\nonumber\\
\label{scheme-vp}\\[-2.ex]
&  + \frac{g}{\varepsilon} (1-|\psi^{k}|^{2}) \psi^{k}.\nonumber
\end{align}
The second step (\ref{projection}) ensures that $|\mathbf{n} ^{k+1}| = 1$ at
each grid point. Note that $|\mathbf{n} ^{*}| \neq0$ in (\ref{projection})
since we consider the case where there are no point defects in the liquid
crystal. The consistency and convergence of the projection method are given in
\cite{E-W}. The method is first order accurate in time and second order
accurate in space due to the first order accuracy of the projection method
(\ref{scheme-n})-(\ref{projection}). To solve the implicit system, we perform
a discrete Fourier transform in the $x$ and $y$ direction. The resulting
tridiagonal systems in the $z$ direction are solved using Gauss elimination.

\begin{figure}
\centering
\begin{tabular}
[c]{ccc}%
\epsfig{file=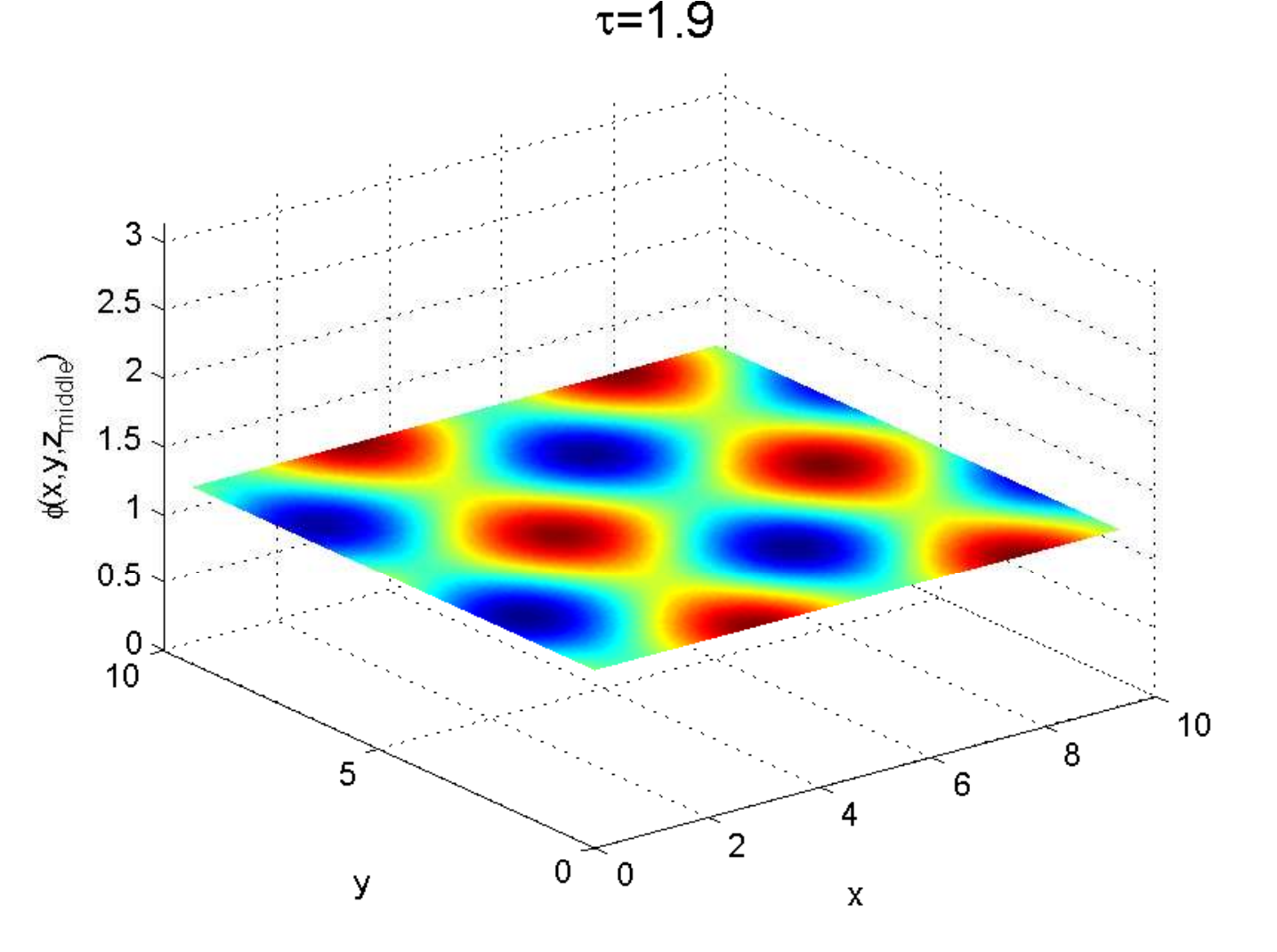,width=0.32\linewidth,clip=} &
\epsfig{file=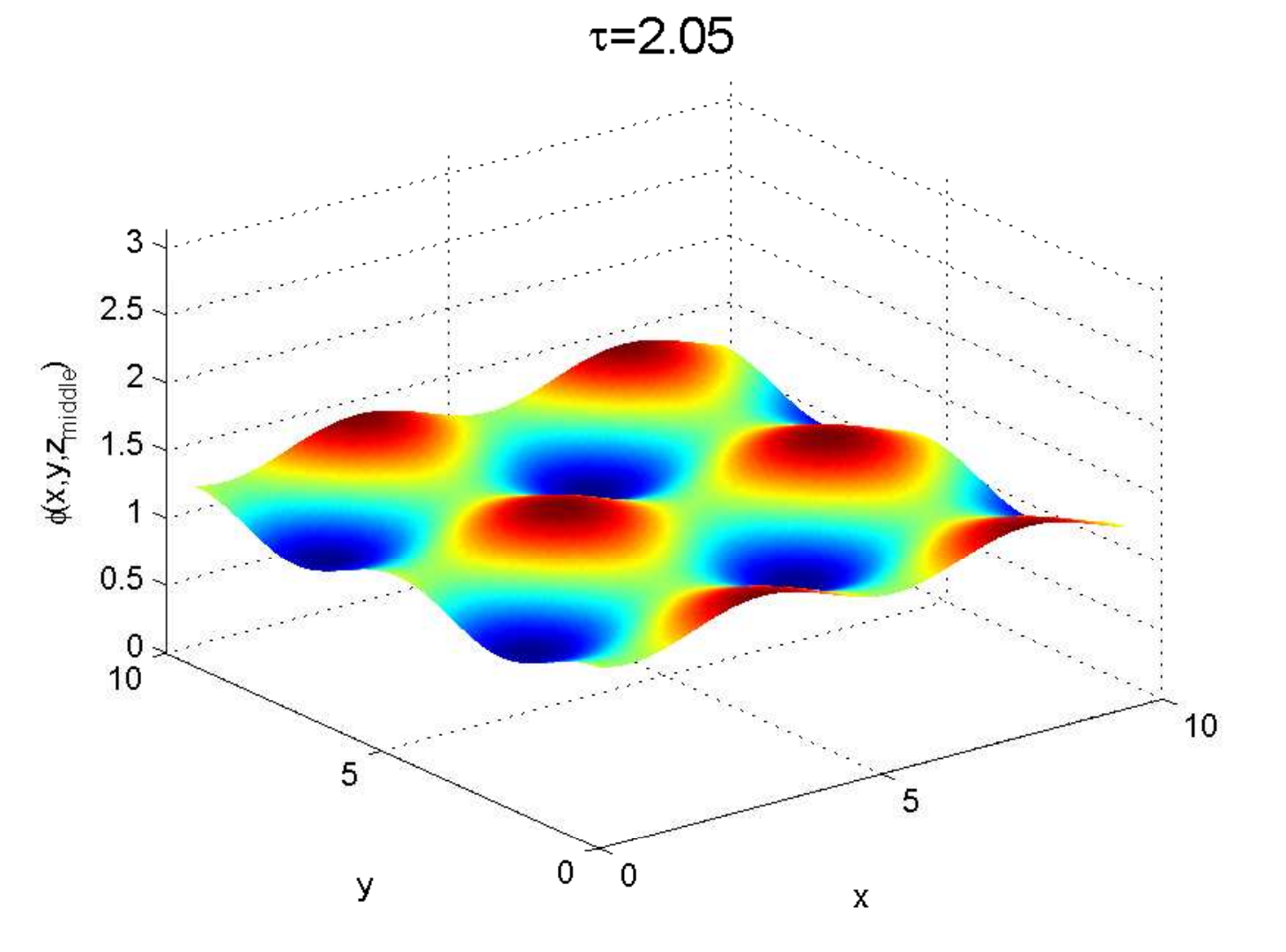,width=0.32\linewidth,clip=}&
\epsfig{file=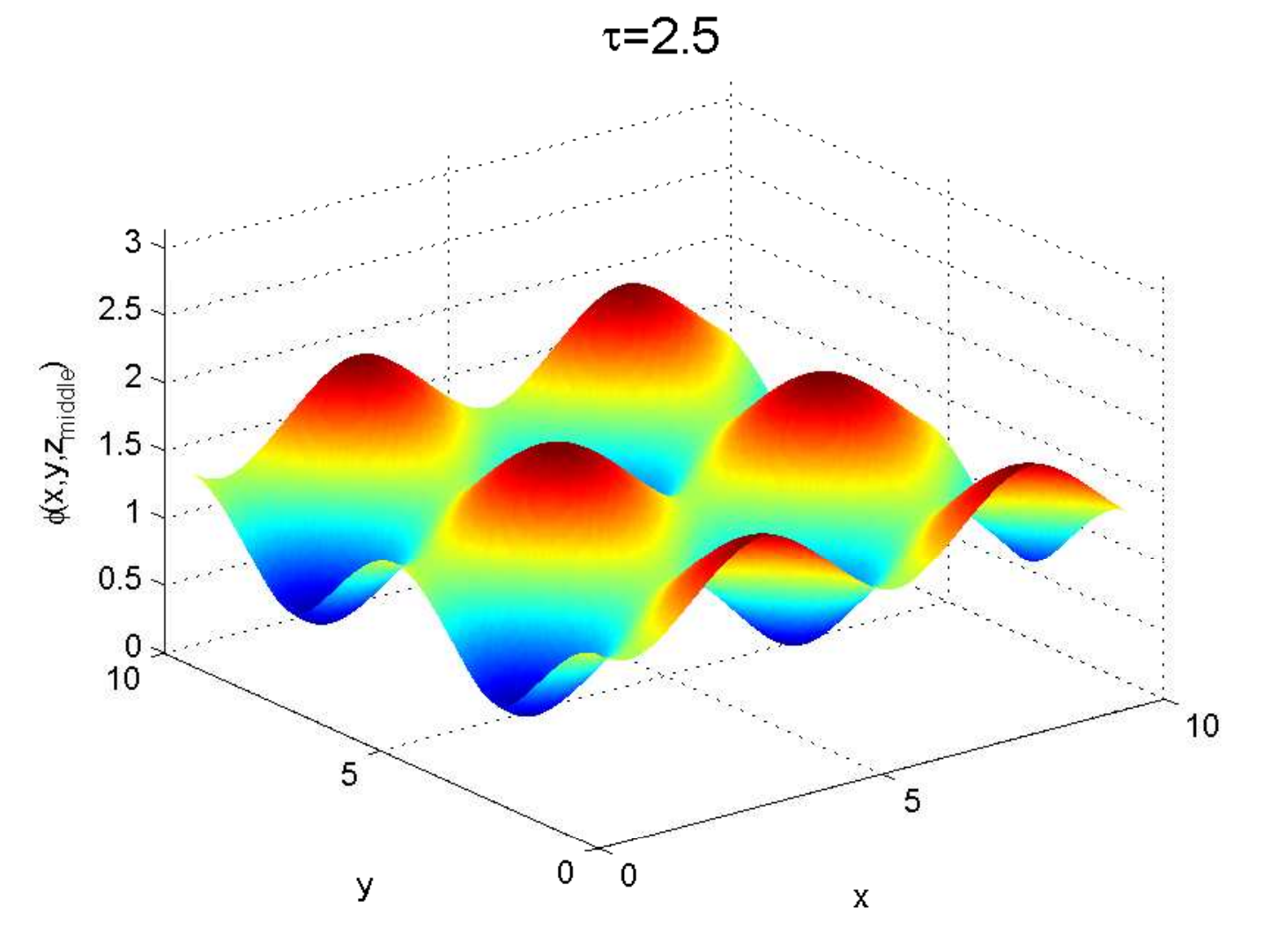,width=0.32\linewidth,clip=}\\
\end{tabular}
\caption{Numerical simulations with Dirichlet boundary conditions on $\psi$ on the bounding plates at various field strengths in the middle of the domain, $z=0$. The onset of layer undulations are observed in the second column. } \label{fig:layerxy}%
\begin{tabular}
[c]{cc}%
~ & ~\\
~ & ~\\
\epsfig{file=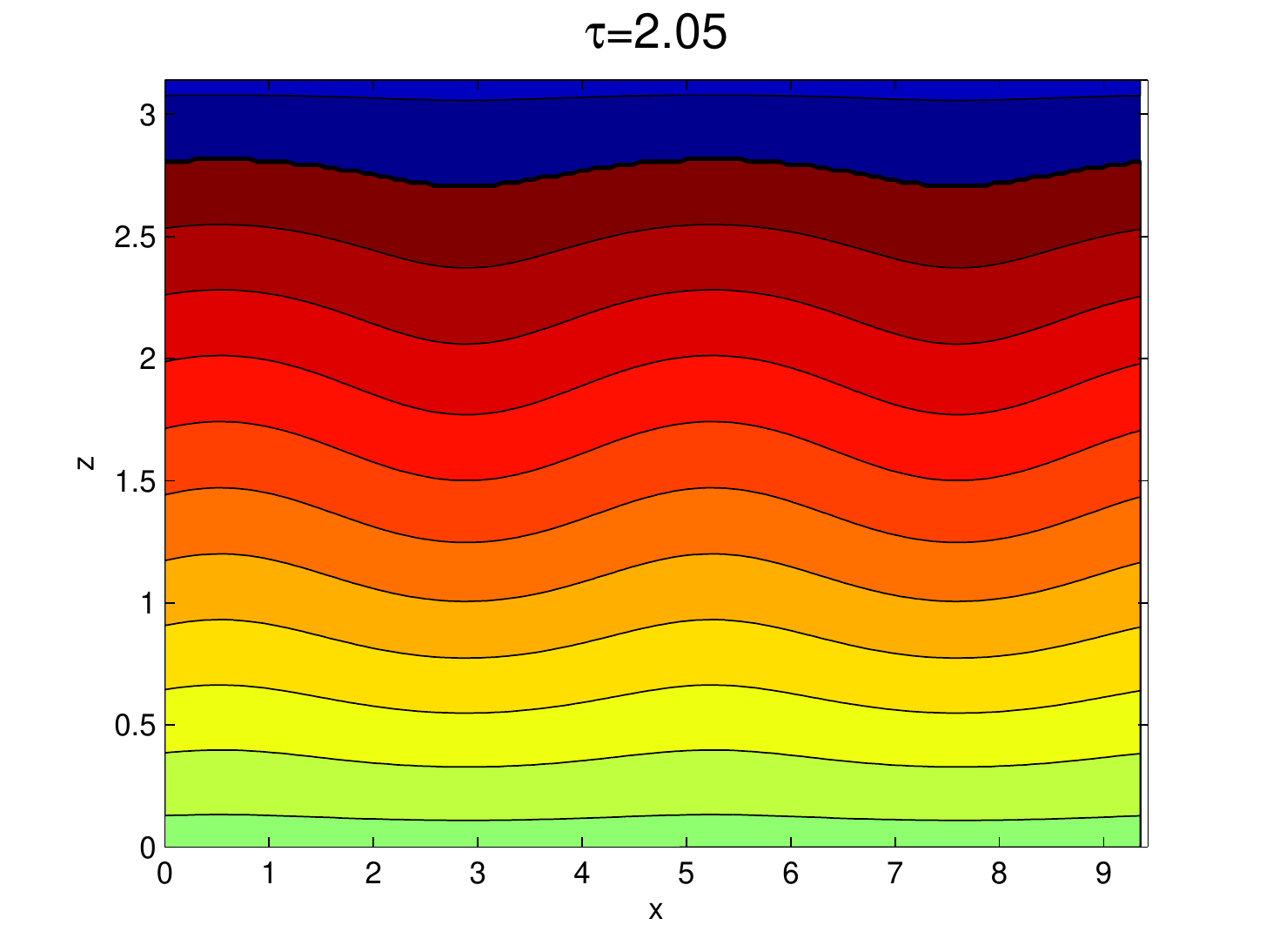,width=0.40\linewidth,clip=} &
\epsfig{file=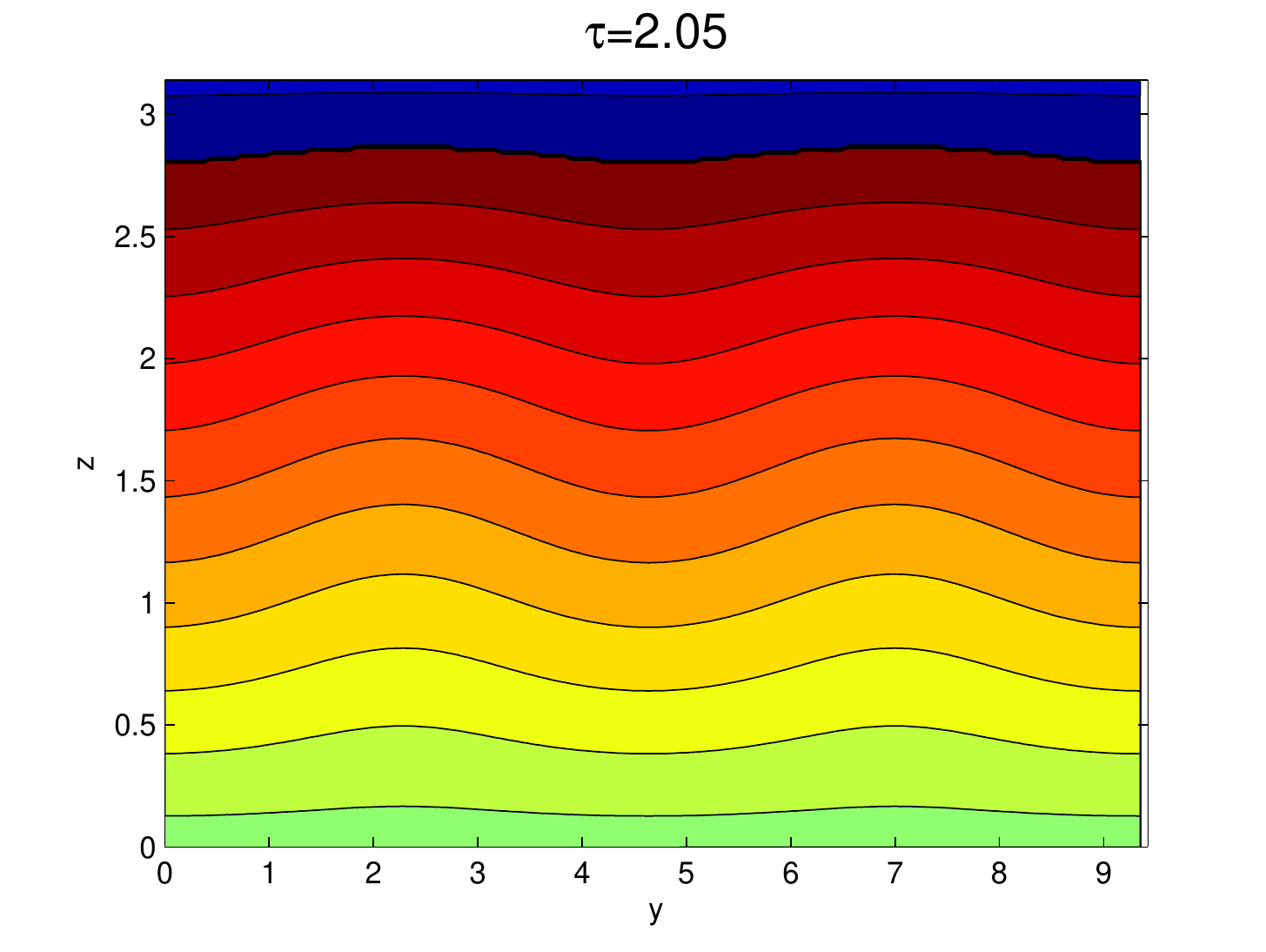,width=0.40\linewidth,clip=}\\
\end{tabular}
\caption{The layer structure on the vertical cross sections, $y=0$ on the left column and $x=0$ on the right column, near the onset of the undulations.  }\label{fig:layerz}%
\end{figure}

We take the domain size $L_{1}= L_{2}= 3 d_{0}$, i.e., in the case of $a=b$, and $\varepsilon= 0.3$. The
number of grid points in the $x$, $y$ and $z$ directions are all $128$. Here
we use parameters
\[
K=0.001, \quad C=0.01,\quad g_{0}=0.5, \quad r=0.25, \quad\mbox{and} \quad
q=10,
\]
and then we obtain from (\ref{parameters})
\[
c = \sqrt{5}, \quad\mbox{and} \quad g=0.25.
\]

We have solved this system in \cite{G-J4} to understand the minimizer at a
various field strength in a two dimensional domain, $\Omega= (0,l)^{2}$ and $n
\in\mathbb{S}^{1}$. Here we consider a three dimensional domain with $n
\in\mathbb{S}^{2}$. In Figure~\ref{fig:layerz} we show the layer structures in the middle of the sample ($z=0$) in response to
various magnetic field strengths $\tau$. In figures, we show the contour maps of $\psi$ since the level sets of
$\psi$ represent the smectic layers. One can see that the undeformed state is
stable for values of $\tau$ below the critical field $\tau_{c} \sim 2 $ as seen in the
left column of Figure~\ref{fig:layerxy}. If $\tau$ increases and reaches $\tau_{c}$,
layer undulations of 2d square patterns occur (the second column of
Figure~\ref{fig:layerxy}). The estimates of the critical field, $\tau_{c}
\sim 2$, is consistent with our analytical result (\ref{first-eigenvalue}). The third column of Figure~
\ref{fig:layerxy} shows that the amplitude of undulations grows in the
increase of the field strengths.
Figure \ref{fig:layerz} depicts the contour map of the phase of $\psi$ in two vertical sections, $y=0$ and $x=0$. As in the classic Helfrich-Hurault theory and result from two dimension \cite{G-J2}, the figures show that the maximum undulations occur in the middle of the domain and layer perturbations decreases as approaching the bounding plates.
\subsection{Planar description}

\label{numerics-planar} In this section, we present numerical simulations of
liquid crystal equilibrium states which minimize (\ref{energy-planar}) to
support analytical results given in section \ref{section:planar}. We consider
the gradient flow (in $L^{2})$ of the energy (\ref{energy-planar}) in
$\mathbb{T}^{2} = (0, 1)^{2} $ with $\delta= 1.5$. The gradient flow equations
are
\begin{equation}\label{GradientFlow2}
\begin{aligned}
\frac{\partial\varphi}{\partial t}  &  = \frac{1}{\varepsilon} \left(
\Delta\varphi- \nabla\cdot n \right)  ,\\
\frac{\partial n}{\partial t}  &  = \Pi_{n} \left(  \varepsilon\Delta n +
\frac{1}{\varepsilon} \left(  \nabla\phi- n\right)  - \frac{1}{\varepsilon
^{\delta}} (n \cdot\mathbf{e}_{3}) \mathbf{e}_{3} \right) ,
\end{aligned}
\end{equation}
where we have used the same notation as in the section
\ref{sec:numerics-complex}.

As initial condition, we consider, for all $(x,y) \in\mathbb{T}^{2} $,
\begin{align*}
\mathbf{n} (x,y,0)  &  = \frac{(\eta u_{1}, \eta u_{2}, 1+ \eta u_{3})}{|(\eta
u_{1}, \eta u_{2}, 1+ \eta u_{3}) |},\\
\varphi(x,y,0)  &  = \eta\varphi_{0},
\end{align*}
where a small number $\eta= 0.1$, $u_{1}, u_{2},u_{3}$ and $\varphi_{0}$ are
arbitrarily chosen.

\begin{figure}[ptb]
\centering
\begin{tabular}
[c]{cc}%
\epsfig{file=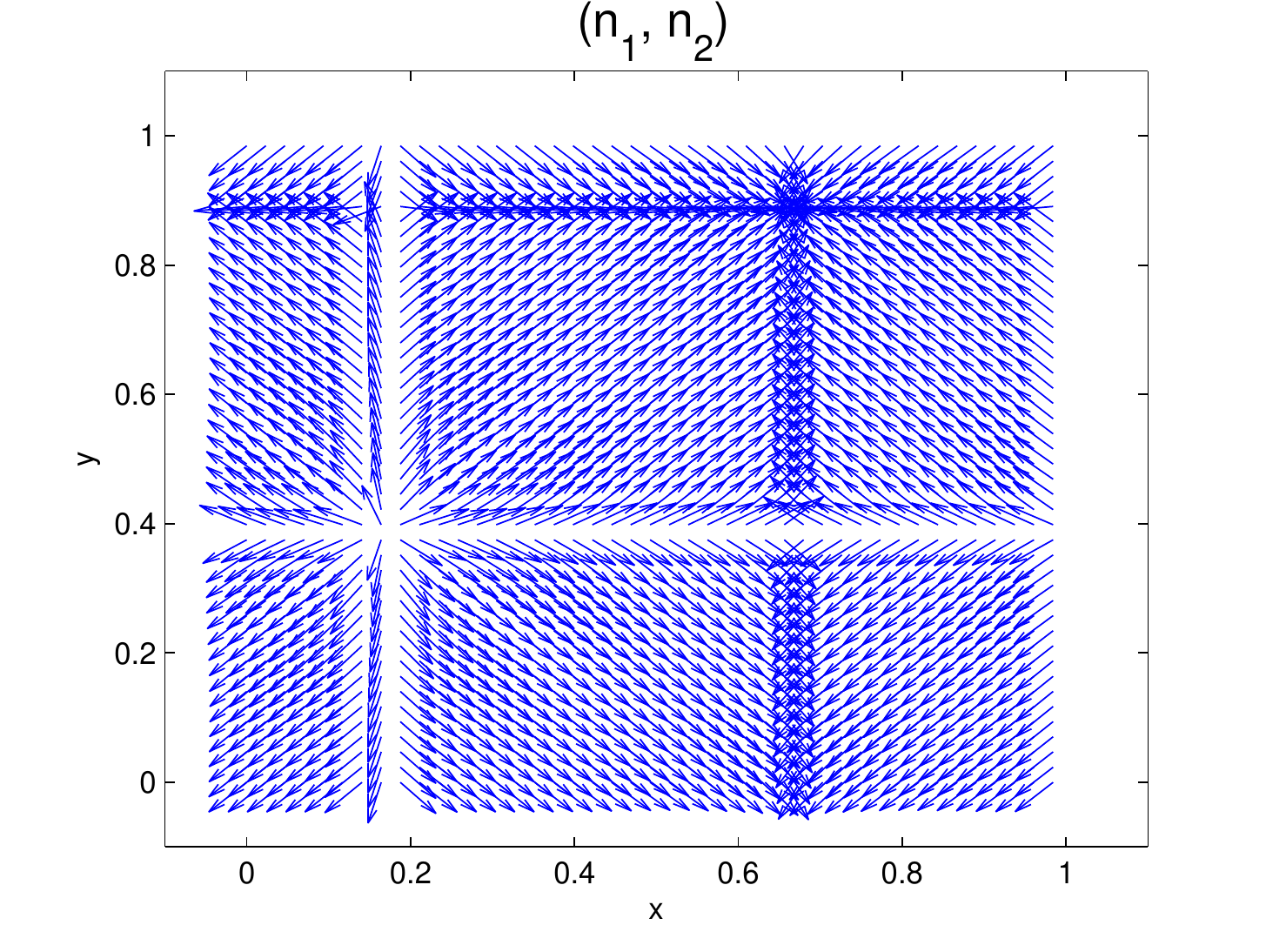,width=0.45\linewidth,clip=} &
\epsfig{file=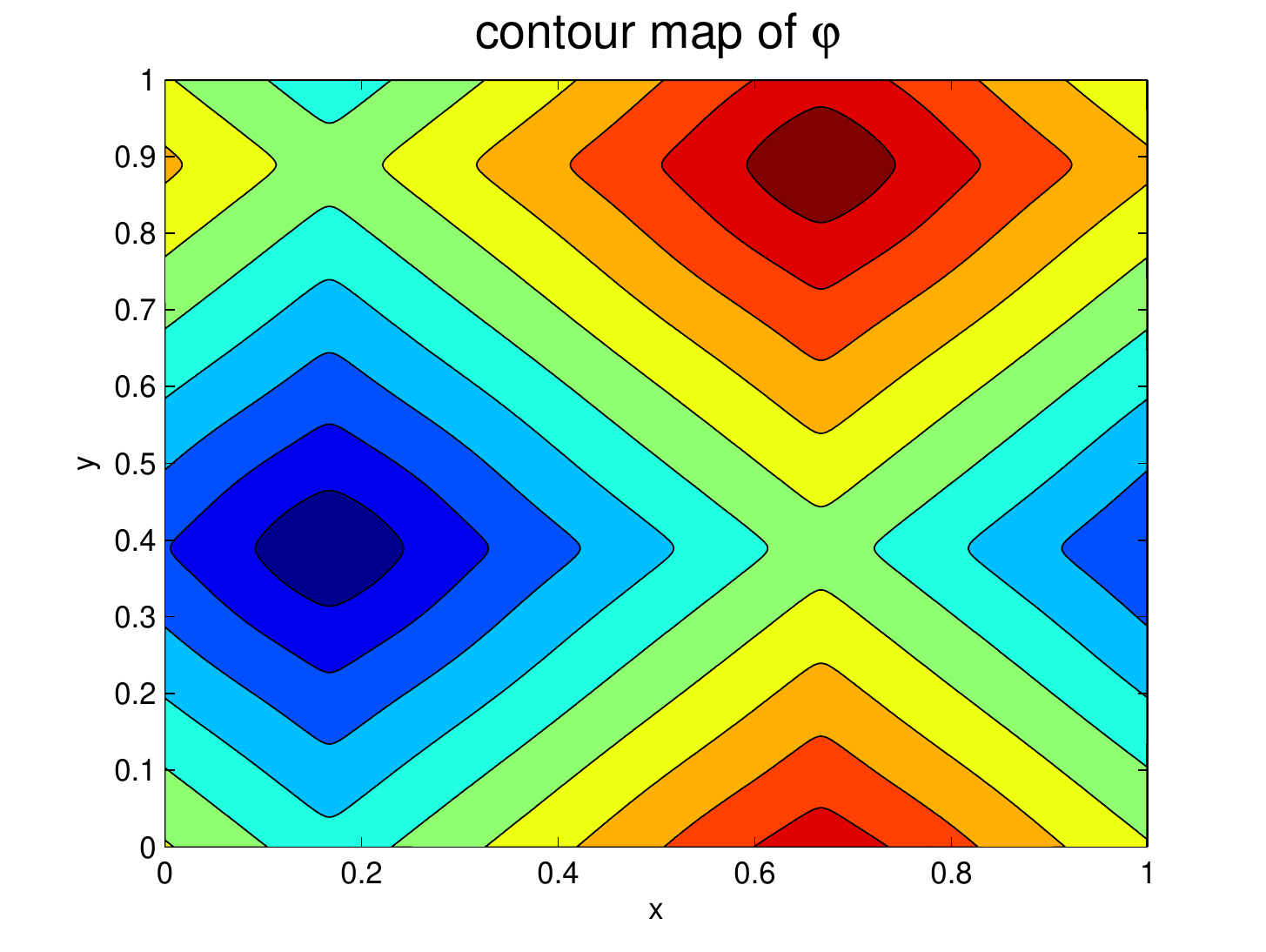,width=0.43\linewidth,clip=}\\
\hskip -0.1cm \textbf{(a)} & \hskip -0.1cm \textbf{(b)}\\[0.2ex]
& \\
\epsfig{file=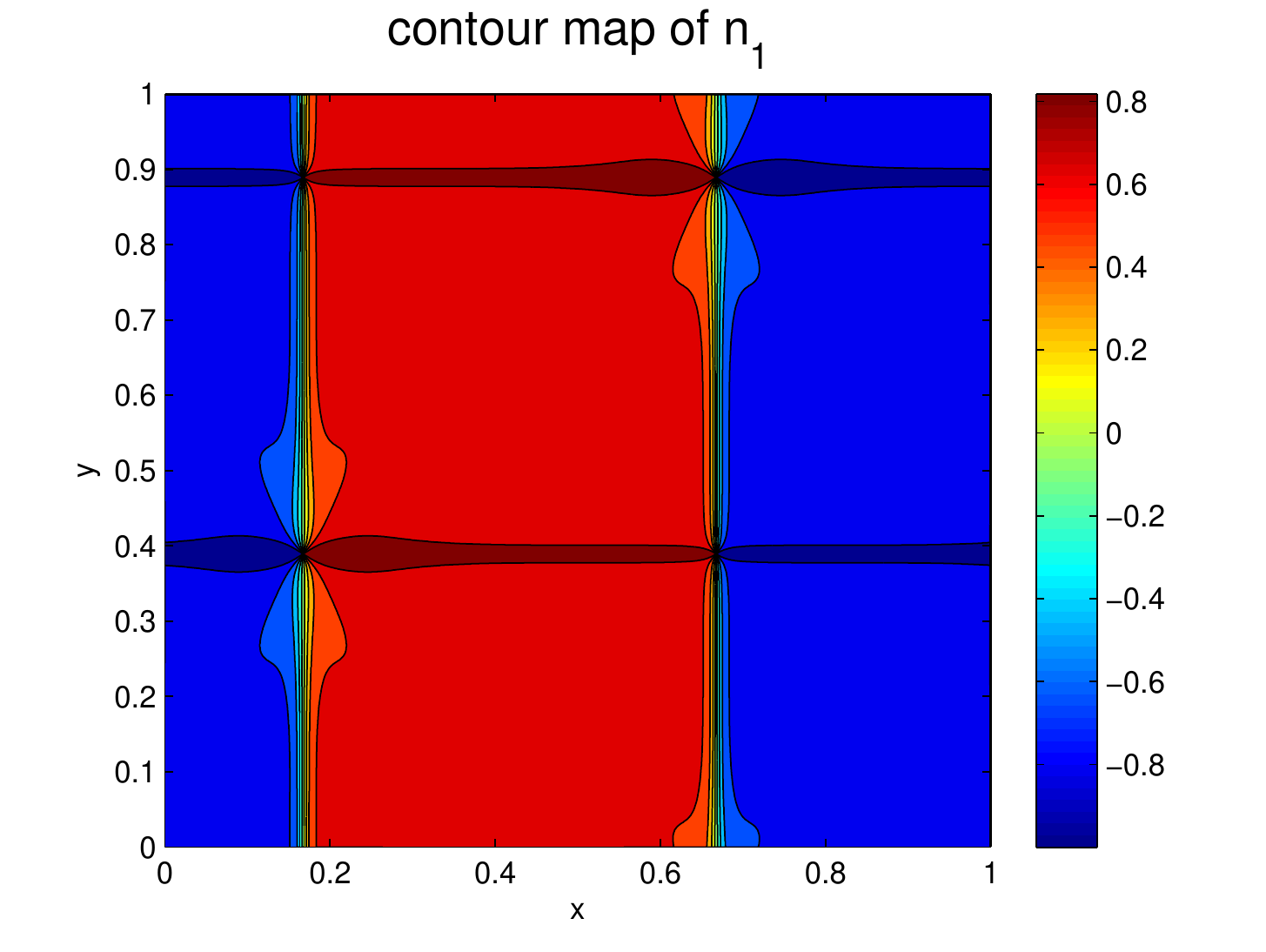,width=0.45\linewidth,clip=} &
\epsfig{file=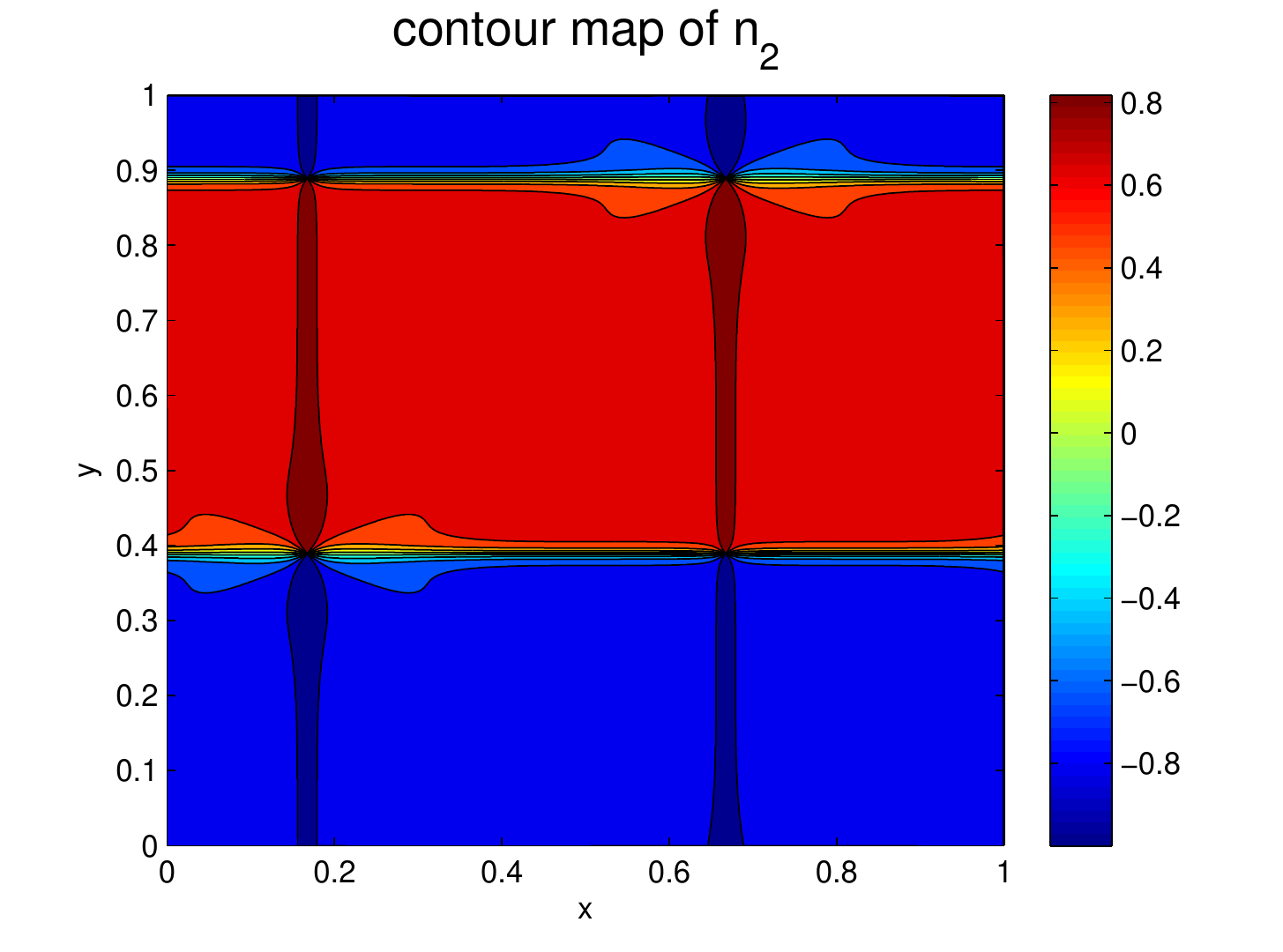,width=0.45\linewidth,clip=}\\
\hskip -0.1cm \textbf{(c)} & \hskip -0.1cm \textbf{(d)}\\[0.2ex]
& \\
&
\end{tabular}
\caption{2D planar configuration: director in the plane. vector field plot of
$\mathbf{n} ^{\parallel}$ and contour map of $\varphi$ in the first row,
contour maps of $n_{1}$ and $n_{2}$ in the second row. }%
\label{fig:planar}%
\end{figure}

We take $\varepsilon= 0.005$ and $1024$ grid points in the $x$ and $y$
direction, which ensures that the transition layers are accurately resolved.

In Figure \ref{fig:planar} (a) and (b) we illustrate the director and layer
configuration and confirm the analytical results in section
\ref{section:planar}. Numerical simulations illustrate 2d square patterns as
an equilibrium state. (See Figures \ref{fig:1d-2d} (a) and \ref{fig:planar}%
(a).) Furthermore, the contour maps of $n_{1}$ and $n_{2}$ in Figure
\ref{fig:planar} (c) and (d) clearly demonstrate the zero mass constraint
given in (\ref{limit}).

The numerical values of energy (\ref{energy-planar}) at the equilibrium state
with various values of $\varepsilon$ are calculated and written in the table
\ref{table:energy}. These values agree with the lower bound given in
(\ref{energy-2d}).

\begin{table}[h]
\caption{Computation of the energy (\ref{energy-planar}) with the various values of $\e$ } \label{table:energy}
\begin{center}
\begin{tabular}{ >{\centering\arraybackslash}m{1.0in}  >{\centering\arraybackslash}m{.5in} >{\centering\arraybackslash}m{.5in} >
{\centering\arraybackslash}m{.5in}  >{\centering\arraybackslash}m{.5in}}
\toprule[0.5pt]
 $\e$ \; & $ \; 0.01 \; $ & $ \; 0.005 \; $ & \; $0.003$ \; & \; $ 0.001$ \; \\
\midrule
Energy (\ref{energy-planar}) & \; 3.50 \;& \;2.82\; &\;2.56\; & \;2.48\; \\
\bottomrule[0.5pt]
\end{tabular}
\end{center}
\end{table}


\end{document}